\documentclass[reqno, 12pt]{amsart}
\usepackage[letterpaper, margin=1in]{geometry}
\usepackage{amsmath, amssymb, amsfonts, amsthm, mathtools, stmaryrd} 
\usepackage{enumitem} 
\setenumerate{label = \textup{(\textit{\roman*})}}
\usepackage[dvipsnames]{xcolor}
\usepackage{tikz}
\usetikzlibrary{positioning}
\usetikzlibrary{calc,decorations.markings,arrows.meta}
\usepackage{tkz-graph}
\usepackage{comment}
\usepackage{bm}
\usepackage{upgreek}
\usepackage[english]{babel}
\usepackage{csquotes}
\usepackage{tikz-cd}
\usepackage{subcaption}
\usepackage{hyperref}
\usepackage{cleveref}
\usepackage{todonotes}
\usepackage{graphicx} 
\usepackage{float}

\usepackage{subcaption}
\usetikzlibrary{positioning}

\usepackage{hyperref}
\hypersetup{
    linkcolor = blue!80!gray,
    citecolor = blue!80!gray,
    urlcolor = blue!80!gray,
    colorlinks = true,
}

\usepackage[style=alphabetic, url=false]{biblatex}
\usepackage{float}
\theoremstyle{definition}
\newtheorem{theorem}{Theorem}[section]
\newtheorem{definition}[theorem]{Definition}
\newtheorem{lemma}[theorem]{Lemma}
\newtheorem{proposition}[theorem]{Proposition}
\newtheorem{corollary}[theorem]{Corollary}
\newtheorem{conjecture}[theorem]{Conjecture}

\newtheorem{remark}[theorem]{Remark}
\newtheorem{convention}[theorem]{Convention}

\definecolor{bettergreen}{RGB}{50, 170, 50}

\definecolor{betterpurple}{RGB}{166, 105, 199}

\newcommand{\D}{\mathbb{D}}

\newcommand{\ZZ}{\mathbb{Z}}
\newcommand{\ocalS}{\overline{X}} 
\newcommand{\calS}{X} 
\newcommand{\ocalCn}{\overline{\mathcal{C}}_n}
\newcommand{\calCn}{\mathcal{C}_n}

\newcommand{\oT}{\overline{T}}

\newcommand{\Chev}{\mathrm{Chev}}

\usepackage{Young}

\title{Multitriangulations on the half-cylinder}
\date{}

\author[Solotko]{Saskia Solotko}
\address{
   Department of Mathematics, Tufts University, Medford, MA
}
\email{saskia.solotko@tufts.edu}

\author[Tung]{Katherine Tung}
\address{
   Department of Mathematics, Harvard University, Cambridge, MA
}
\email{katherinetung@college.harvard.edu}

\author[Yang]{Mengyuan Yang}
\address{
   Department of Mathematics, Swarthmore College, Swarthmore, PA
}
\email{myang7@swarthmore.edu}

\author[Zhang]{Yuchong Zhang}
\address{
   Department of Mathematics, University of Michigan, Ann Arbor, MI
}
\email{zongxun@umich.edu}

\begin{document}

\maketitle

\begin{abstract}

We prove that the simplicial complex $\Delta _{\calCn,2}$ is pure and a weak pseudomanifold of dimension $2(n-1)$, where $\Delta _{\calCn,2}$ is the simplicial complex associated with $2$-triangulations on the half-cylinder with $n$ marked points. This result generalizes the work of Vincent Pilaud and Francisco Santos for polygons and resolves a conjecture of Mathias Lepoutre and Vincent Pilaud for $k=2$. To achieve this, we show that $2$-triangulations on the half-cylinder decompose as complexes of star polygons, and that $2$-triangulations on the half-cylinder are in bijection with $2$-triangulations on the $4n$-gon invariant under rotation by $\pi/2$ radians. Building on work by Vincent Pilaud and Christian Stump, we also introduce chevron pipe dreams, a new combinatorial model that more naturally captures the symmetries of $k$-triangulations.

\end{abstract}


\section{Introduction}\label{sec:intro}

A $k$-triangulation, or multitriangulation of order $k$, on the convex $n$-gon is a maximal set of edges (namely, internal diagonals) such that no $k+1$ of them mutually intersect inside the polygon. See \Cref{3periodic2tri} for an example of a $2$-triangulation on the $12$-gon. When $k=1$, a $1$-triangulation corresponds to the usual notion of an ideal triangulation.

 \begin{figure}[H]

\begin{tikzpicture}[auto=center,scale=2.7]

    \draw[line width=0.5mm,spink] (0.5, 0.866) -- (0, -1);
    \draw[line width=0.5mm,spink] (0, 1) -- (-0.5, -0.866);
    \draw[line width=0.5mm,spink] (1, 0) -- (-0.866, 0.5);
    \draw[line width=0.5mm,spink] (-1, 0) -- (0.866, -0.5) ;

    \draw[line width=0.5mm,syellow] (1, 0) -- (-0.5, 0.866);
    \draw[line width=0.5mm,syellow] (0, 1) -- (-0.866, -0.5);
    \draw[line width=0.5mm,syellow] (-1, 0) -- (0.5, -0.866);
    \draw[line width=0.5mm,syellow] (0, -1) -- (0.866, 0.5);

    \draw[line width=0.5mm,sorange] (0, 1) -- (0, -1);
    \draw[line width=0.5mm,sorange] (1, 0) -- (-1, 0);

    \draw[line width=0.5mm,steal] (0, 1) -- (1, 0);
    \draw[line width=0.5mm,steal] (1, 0) -- (0, -1);
    \draw[line width=0.5mm,steal] (0, -1) -- (-1, 0);
    \draw[line width=0.5mm,steal] (-1, 0) -- (0, 1);

    \draw[line width=0.5mm,sgrey] (0.866, 0.5) -- (0.866, -0.5);
    \draw[line width=0.5mm,sgrey] (-0.5, 0.866) -- (0.5, 0.866);
    \draw[line width=0.5mm,sgrey] (-0.866, 0.5) -- (-0.866, -0.5);
    \draw[line width=0.5mm,sgrey] (-0.5, -0.866) -- (0.5, -0.866);

    \draw[line width=0.5mm,sgrey] (0.5, 0.866) -- (1, 0);
    \draw[line width=0.5mm,sgrey] (-0.866, 0.5) -- (0, 1);
    \draw[line width=0.5mm,sgrey] (-0.5, -0.866) -- (-1, 0);
    \draw[line width=0.5mm,sgrey] (0, -1) -- (0.866, -0.5);

    \draw[line width=0.5mm,sgrey] (0, 1) -- (0.866, 0.5);
    \draw[line width=0.5mm,sgrey] (-1, 0) -- (-0.5, 0.866);
    \draw[line width=0.5mm,sgrey] (0, -1) -- (-0.866, -0.5);
    \draw[line width=0.5mm,sgrey] (0.5, -0.866) -- (1, 0);

    \draw[line width=0.5mm,sgrey] (0.866, 0.5) -- (0.5, 0.866);

    \draw[line width=0.5mm,sgrey] (0.5, 0.866) -- (0, 1);

    \draw[line width=0.5mm,sgrey] (0, 1) -- (-0.5, 0.866);

    \draw[line width=0.5mm,sgrey] (-0.5, 0.866) -- (-0.866, 0.5);

    \draw[line width=0.5mm,sgrey] (-0.866, 0.5) -- (-1, 0);

    \draw[line width=0.5mm,sgrey] (-1, 0) -- (-0.866, -0.5);

    \draw[line width=0.5mm,sgrey] (-0.866, -0.5) -- (-0.5, -0.866);

    \draw[line width=0.5mm,sgrey] (-0.5, -0.866) -- (0, -1);

    \draw[line width=0.5mm,sgrey] (0, -1) -- (0.5, -0.866);

    \draw[line width=0.5mm,sgrey] (0.5, -0.866) -- (0.866, -0.5);

    \draw[line width=0.5mm,sgrey] (0.866, -0.5) -- (1, 0);

    \draw[line width=0.5mm,sgrey] (1, 0) -- (0.866, 0.5);

    \draw[thick,red] (0.866, 0.5) node  {$\bullet$};
    \draw[thick,Green] (0.5, 0.866) node {$\bullet$};
    \draw[thick,blue] (0, 1) node {$\bullet$};
    \draw[thick,red] (-0.5, 0.866) node {$\bullet$};
    \draw[thick,Green] (-0.866, 0.5) node {$\bullet$};
    \draw[thick,blue] (-1, 0) node {$\bullet$};
    \draw[thick,red] (-0.866, -0.5) node {$\bullet$};
    \draw[thick,Green] (-0.5, -0.866) node {$\bullet$};
    \draw[thick,blue] (0, -1) node {$\bullet$};
    \draw[thick,red] (0.5, -0.866) node {$\bullet$};
    \draw[thick,Green] (0.866, -0.5) node {$\bullet$};
    \draw[thick,blue] (1, 0) node {$\bullet$};
\end{tikzpicture}

\caption{A 
$2$-triangulation on the $12$-gon invariant under rotation by $\pi/2$ radians.  \label{3periodic2tri}}
\end{figure}
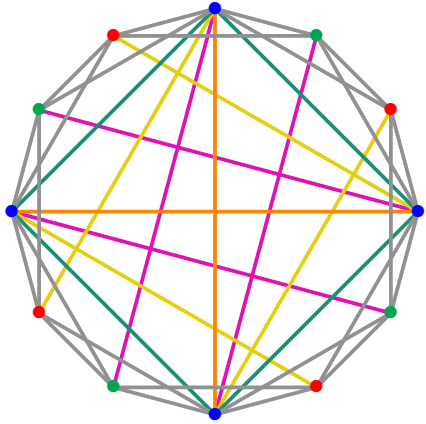

Jakob Jonsson, Vincent Pilaud, Francisco Santos, Christian Stump, and many other authors have proved various structural results about $k$-triangulations on polygons:

\begin{enumerate}
    \item Each $k$-triangulation $T$ on the $n$-gon can be regarded as a simplicial complex $\delta^T_{n,k}$, where each facet is a $k$-star \cite{MR2721464}. 

    \item All $k$-triangulations on the $n$-gon have the same number of edges, which is equal to $k(2n-2k-1)$ \cite{CP92, NAKAMIGAWA2000271, DKM00}. 

    \item Every $k$-relevant edge $e$ of a $k$-triangulation $T$ can be ``flipped" to a unique edge $f$ such that $(T\setminus \{e\})\cup \{f\}$ is a $k$-triangulation \cite{NAKAMIGAWA2000271, MR2721464}. 

    \item There is a bijection between $k$-triangulations on the $n$-gon and reduced pipe dreams for the permutation $\pi_{n,k}=[1,\dots, k,n-k,n-k-1,\dots, k+1]$ \cite{PilaudThesis, STUMP20111794, PP12}.

    \item The simplicial complex $\Delta_{n,k}$, whose vertices are the $k$-relevant edges of the $n$-gon and whose facets correspond to $k$-triangulations, is a vertex-decomposable sphere \cite{Jonssonfake}.

\end{enumerate}

Further, $k$-triangulations on polygons admit many natural interpretations, such as

\begin{enumerate}
    \item $k$-compatible split systems \cite{DKKM01},
    \item pseudoline arrangements on the Möbius strip \cite{PilaudThesis,PP12}, 
    \item multi-cluster complexes of type A \cite{CeballosLabbeStump}, and
    \item bases of the Pfaffian variety $\mathcal{P}f_k(n)$ of antisymmetric matrices of rank $\le 2k$ \cite{CS24}.
\end{enumerate}

In this paper, we generalize the notion of $k$-triangulations on polygons to $k$-triangulations on surfaces with marked points.

\begin{definition}[$k$-triangulations]
    Let $\calS$ be a surface, let $\overline{\calS}$ be its universal cover, and let $\pi: \overline{\calS} \to 
    \calS$ be an associated covering map. A \emph{(weak) $k$-triangulation} on $\calS$ is a maximal subset of edges $e$ so that the set of preimages $\pi^{-1}(e) \in \overline{\calS}$ does not contain any \emph{$(k+1)$-crossings} (namely, no $k+1$ of them mutually intersect inside the polygon).
\end{definition}

\Cref{fig:univ-cover} is an example of a $2$-triangulation on the half-cylinder with $3$ marked points.

\begin{figure}[H]
    \centering 
    \includegraphics[width=0.8\linewidth]{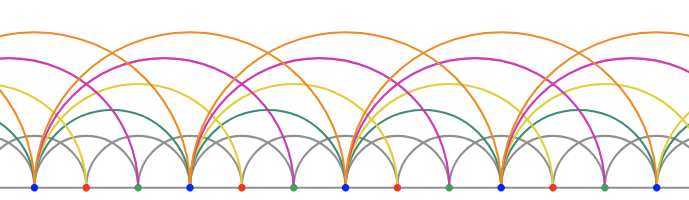}
    \caption{A $2$-triangulation on the half-cylinder with $3$ marked points, lifted to the universal cover.
    \label{fig:univ-cover}}
    
\end{figure}

Our main result is the following.
\begin{theorem}\label{thm:main}
    The flip graph of $2$-triangulations on $\calCn$, the half-cylinder with $n$ marked points, is regular. Equivalently, the associated simplicial complex $\Delta_{\calCn,2}$ is pure and a weak pseudomanifold. 
\end{theorem}


\subsection{Amending a previous definition of $k$-triangulations on surfaces.} For readers familiar with the $k$-triangulation literature, we summarize in the next six paragraphs some previous attempts to define $k$-triangulations on surfaces with marked points and explain the advantages of our definition. In 2019, Mathias Lepoutre and Vincent Pilaud proposed generalizing $k$-triangulations on polygons to arbitrary surfaces by studying $k$-triangulations on their universal covers \cite[Part III]{PHDThesis}. Their $k$-triangulations are ``the projections of \textbf{periodic maximal} $(k+1)$-crossing-free edge sets on the universal cover" \cite[Definition~1.1.4~and~Definition~1.2.2~of~Part~III]{PHDThesis}. We call such objects \emph{strong $k$-triangulations}.

In contrast, our $k$-triangulations are ``the projections of \textbf{maximal periodic} $(k+1)$-crossing-free edge sets on the universal cover." We call these objects \emph{weak $k$-triangulations}.\footnote{We conceived of this definition before discovering \cite{PHDThesis}.} 

In \cite[Remark~12~of~Part~III]{PHDThesis}, it is conjectured that on any surface, the notions of strong $k$-triangulation and weak $k$-triangulation are equivalent. However, this conjecture is incorrect. In fact, on many naturally arising surfaces, weak $k$-triangulations exist and strong $k$-triangulations do not. We decided to study weak $k$-triangulations because they have a richer structure than strong $k$-triangulations in the following two ways. 

First, 
weak $1$-triangulations 
always coincide with \emph{ideal triangulations} (maximal collection of non-null-homotopic arcs that are pairwise non-homotopic) whereas strong $1$ triangulations do not. For example, $\calCn$ has ideal triangulations but no strong $1$-triangulations. 

\begin{remark}\label{remark:revisedstrong}
    There is a problem with the proof of \cite[Proposition~1.3.3~of~Part~III]{PHDThesis}, namely that there are no strong $k$-triangulations on the half-cylinder for any $k$. (Though \cite[Figure~1.11~of~Part~III]{PHDThesis} seems to suggest otherwise, the figure does not correctly depict a strong $k$-triangulation, as a single edge connecting $u_0$ and $u_4$ can be added without creating a $4$-crossing). 
\end{remark}

Second, weak $k$-triangulations lend themselves to an underexplored analog of \cite[Theorem~4.1]{MR2721464}, which we call the Star Decomposition Theorem for polygons. 
While the Star Decomposition Theorem for the strong $k$-triangulations on any finite-type surface follows from the corresponding result for infinite tidy polygons (see \cite[Section~1.2~of~Part~III]{PHDThesis}), the Star Decomposition Conjecture for the weak $k$-triangulations is far more subtle. 
In this paper, we prove the Star Decomposition Conjecture for weak $k$-triangulations on the half-cylinder when $k=2$ (\Cref{2starparty}), and we regard this result as the main technical contribution of our work. 

The definition of $k$-triangulations implicitly used by Pilaud is ``the projections of a periodic maximal [$(k+1)$-crossing-free sets on the infinite polygon that contain no edge which, when extended periodically, would create a $(k+1)$-self-crossing on the universal cover]" \cite{vppc}. We further prove that on the half-cylinder, the weak $k$-triangulations are precisely the $k$-triangulations they implicitly used when $k=2$ (\Cref{maximalityfromstars}).

\begin{convention}
    In the rest of the paper, $k$-triangulations refer to weak $k$-triangulations unless otherwise specified.
\end{convention}

This paper is organized as follows. In \Cref{sec:back}, we define necessary terms 
including $k$-triangulations, $k$-stars, angles, and $k$-relevant edges. In \Cref{sec:multi}, we prove a number of properties about multitriangulations on the half-cylinder. 
A main result is \Cref{2starparty}, which serves as a generalization of the Star Decomposition Theorem for $n$-gons (\cite[Theorem~4.1]{MR2721464}). In \Cref{sec:bij}, we give a bijection between $2$-triangulations on the half-cylinder with $n$ marked points and $n$-periodic $k$-triangulations on the $2kn$-gon when $k=2$. 
In \Cref{sec:chev}, we define chevron pipe dreams, a symmetric analogue of \emph{reduced pipe dreams} \cite{BB93, FK93} 
that captures central symmetries of $k$-triangulations through mirror symmetry and periodicity conditions. See \Cref{fig:chevron} for an example.
Finally, in Section \ref{sec:regular}, we prove our main theorem (\Cref{thm:main}) using the aforementioned results.

\begin{figure}[H]
    \centering
    \begin{mosaic}{0.5}
    \& \& \& \& \& \tiF\& \& \& \& \\
    \& \& \& \& \tiF\& \Steal \& \& \& \& \\
    \& \& \& \tiF\& \tiX\& \Syellow\& \& \& \& \\
    \& \& \tiF\& \tiX\& \tiX\& \Spink\& \& \& \& \\
    \& \tiF\& \Steal\& \tiX\& \tiX\& \Sorange\& \Spink\& \Syellow\& \Steal\& \tiJ \\
    \tiF\& \tiX\& \Syellow\& \tiX\& \tiX\& \tiX\& \tiX\& \tiX\& \tiJ \\
    \tiX\& \tiX\& \Spink\& \tiX\& \tiX\& \tiX\& \tiX\& \tiJ \\
    \tiX\& \tiX\& \Sorange\& \Spink\& \Syellow\& \Steal\& \tiJ \\
    \tiX\& \tiX\& \tiX\& \tiX\& \tiX\& \tiJ \\
    \tiX\& \tiX\& \tiX\& \tiX\& \tiJ \\
\end{mosaic}
    \caption{\label{fig:cpd2-flipped} The chevron pipe dream corresponding to the $2$-triangulation from \Cref{3periodic2tri}. Chevron pipe dreams are defined in Definition \ref{chevronconstruction}.}\label{fig:chevron}
    
\end{figure}

\section{Background}\label{sec:back}

Generalizing the $k$-triangulations on a convex $n$-gon, we define $k$-triangulations on the half-cylinder with $n$ marked points. Let $\calCn$ denote the annulus with $n$ marked points $\alpha_{[0]},\dots,\alpha_{[n-1]}$ on one of its boundary components, where subscripts are congruence classes modulo $n.$ We refer to $\calCn$ as a \emph{half-cylinder} or an \emph{$(n,0)$-annulus}.

The universal cover of $\calCn,$ denoted $\ocalCn,$ has vertices $\{\alpha_i\mid i\in \ZZ\}.$ Taking any vertex $v=\alpha_{[i]}\in \calCn,$ we define $v+j\coloneq \alpha_{[i+j]}$ for any integer $j.$ An \emph{edge} $e$ on $\ocalCn$ is defined by the isotopy class of topological curves connecting $u,v$ and is denoted by $[u,v]$ or $[v,u].$ Similarly, we define $e+n\coloneq [u+n,v+n].$ The covering map $\pi:\ocalCn\rightarrow \calCn,$ applied to vertices and edges is given by $\pi(\alpha_{[i]})=\alpha_i$ and $\pi([\alpha_{[i]},\alpha_{[j]}])=[\alpha_i,\alpha_j].$ We say an edge $e$ on $\ocalCn$ is a \emph{translation} of an edge $f$ if $e=f+\ell n$ for some $\ell\in \ZZ,$ or equivalently if $\pi(e)=\pi(f).$ Typically, we do not refer to vertices of $\calCn$ or $\ocalCn$ using these subscripts determined by their position.

There is a natural order of edges on $\ocalCn$ where $u<v$ if and only if $u=v+\ell$ for some $\ell\in \ZZ^{+}$ induced from a counterclockwise orientation on $\ocalCn.$ When considering a subset of vertices of $\ocalCn$, we additionally have a \emph{cyclic order} $\prec$ given by the counterclockwise cyclic order on the first polygon, as defined in \cite{PHDThesis}. On $\ocalCn,$ two edges $e_1$ and $e_2$ are said to \emph{intersect} if all representatives of the isotopy classes of $e_1$ and $e_2$ intersect, or equivalently, $a_1\prec b_1 \prec a_2 \prec b_2$ for some labeling of vertices $e_1=[a_1,b_1],e_2=[a_2,b_2].$ A \emph{$k$-crossing} on $\ocalCn$ is a set of edges $E=\{e_1,e_2,\dots, e_k\}$ that pairwise intersect.

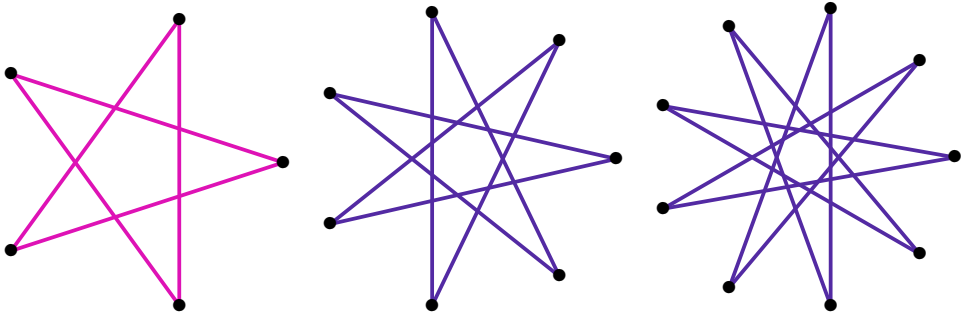
\begin{figure}[H]
\begin{tikzpicture}[scale=2]
    \draw[line width=0.5mm,spink] (0.309, 0.951) -- (0.309, -0.951);
    \draw[line width=0.5mm,spink] (0.309, 0.951) -- (-0.809, -0.588);
    \draw[line width=0.5mm,spink] (-0.809, -0.588) -- (1, 0);
    \draw[line width=0.5mm,spink] (1, 0) -- (-0.809, 0.588);
    \draw[line width=0.5mm,spink] (-0.809, 0.588) -- (0.309, -0.951);
    
    \draw[thick] (0.309, 0.951) node  {$\bullet$};
    \draw[thick] (-0.809, 0.588) node {$\bullet$};
    \draw[thick] (-0.809, -0.588) node {$\bullet$};
    \draw[thick] (0.309, -0.951) node {$\bullet$};
    \draw[thick] (1, 0) node {$\bullet$};
\end{tikzpicture} \begin{tikzpicture}[scale=2]
    \draw[line width=0.5mm,spurple] (0.623, 0.782) -- (-0.901, -0.434);
    \draw[line width=0.5mm,spurple] (-0.901, -0.434) -- (1, 0);
    \draw[line width=0.5mm,spurple] (1, 0) -- (-0.901, 0.434);
    \draw[line width=0.5mm,spurple] (-0.901, 0.434) -- (0.623, -0.781);
    \draw[line width=0.5mm,spurple] (0.623, -0.781) -- (-0.223, 0.974);
    \draw[line width=0.5mm,spurple] (-0.223, 0.974) -- (-0.223, -0.975);
    \draw[line width=0.5mm,spurple] (-0.223, -0.975) -- (0.623, 0.782);
    
    \draw[thick] (0.623, 0.782) node  {$\bullet$};
    \draw[thick] (-0.223, 0.974) node  {$\bullet$};
    \draw[thick] (-0.901, 0.434) node  {$\bullet$};
    \draw[thick] (-0.901, -0.434) node  {$\bullet$};
    \draw[thick] (-0.223, -0.975) node  {$\bullet$};
    \draw[thick] (0.623, -0.781) node  {$\bullet$};
    \draw[thick] (1, 0) node  {$\bullet$};
\end{tikzpicture} \begin{tikzpicture}[scale=2]
    \draw[line width=0.5mm,spurple] (0.766, 0.643) -- (-0.940, -0.342);
    \draw[line width=0.5mm,spurple] (-0.940, -0.342) -- (1, 0);
    \draw[line width=0.5mm,spurple] (1, 0) -- (-0.940, 0.342);
    \draw[line width=0.5mm,spurple] (-0.940, 0.342) -- (0.766, -0.643);
    \draw[line width=0.5mm,spurple] (0.766, -0.643) -- (-0.5, 0.866);
    \draw[line width=0.5mm,spurple] (-0.5, 0.866) -- (0.174, -0.985);
    \draw[line width=0.5mm,spurple] (0.174, -0.985) -- (0.174, 0.985);
    \draw[line width=0.5mm,spurple] (0.174, 0.985) -- (-0.5, -0.866);
    \draw[line width=0.5mm,spurple] (-0.5, -0.866) -- (0.766, 0.643);
    
    \draw[thick] (0.766, 0.643) node  {$\bullet$};
    \draw[thick] (0.174, 0.985) node  {$\bullet$};
    \draw[thick] (-0.5, 0.866) node  {$\bullet$};
    \draw[thick] (-0.940, 0.342) node  {$\bullet$};
    \draw[thick] (-0.940, -0.342) node  {$\bullet$};
    \draw[thick] (-0.5, -0.866) node  {$\bullet$};
    \draw[thick] (0.174, -0.985) node  {$\bullet$};
    \draw[thick] (0.766, -0.643) node  {$\bullet$};
    \draw[thick] (1, 0) node  {$\bullet$};

\end{tikzpicture}
\caption{A $2$-star in a pentagon, a $3$-star in a heptagon, and a $4$-star in a nonagon.\label{fig:kstar}}
\end{figure}

Take a subset $\{z_0,\dots, z_{2k}\}$ of vertices on $\ocalCn$ ordered such that
\[
z_0\prec z_1 \prec\cdots\prec z_{2k}.
\]
A \emph{$k$-star polygon}, or a \emph{$k$-star} $S$ on this subset, as defined in \cite{MR2721464} and \cite{PHDThesis}, contains all vertices $z_i$ and edges $[z_i,z_{i+k}]$ on $\ocalCn$ with subscripts reduced modulo $n.$ 
We use these definitions of $k$-crossings and $k$-stars on $\ocalCn$ to define analogously $k$-crossings and $k$-stars on $\calCn.$

As a generalization of the $k$-stars and the $k$-triangulations on $n$-gons, we introduce the following definitions.

\begin{definition}
A \emph{$k$-star} on $\calCn$ is the projection $\pi(S)$ of a $k$-star $S$ on $\ocalCn.$
\end{definition}

\begin{definition}
A \emph{k-triangulation} $T$ on $\calCn$ is a maximal set of edges such that the lifting $\pi^{-1}(T)$ is $(k+1)$-crossing-free. 
\end{definition}


We additionally define a \emph{$k$-triangulation} $\oT$ on $\ocalCn$ to be a maximal set of \emph{$n$-periodic} edges on $\ocalCn,$ that is, $v\in \oT \Longleftrightarrow v+n\in \oT.$ Equivalently, $\pi(\oT)$ is a $k$-triangulation on $\calCn$.

For vertices $u,v\in \ocalCn,$ we say $v-u=j$ if $u+j=v.$ The \emph{length} of an edge $[u,v]\in \ocalCn$ is defined by $|v-u|.$ We also say $\pi([v,u])$ has \emph{length} $|v-u|.$ 

An edge $[u,v]$ of $\ocalCn$ can appear in a $(k+1)$-crossing only if $|v-u|>k.$ Thus, each $k$-triangulation $T$ (respectively, $\oT$) contains all edges of length $\leq k.$ For simplicity, we adopt the following convention.

\begin{convention}\label{convention:exclude_short_edge}
    In this paper, when discussing a $k$-triangulation $T$ on $\calCn$ (respectively, $\oT$ on $\ocalCn$), we exclude all the edges of length $<k$.
\end{convention}

As we show in \Cref{length>kn} later, a $k$-triangulation $T$ (respectively, $\oT$) contains no edges of length $>kn.$  
These observations inform us to introduce the following definitions, analogous to those in \cite{PHDThesis} and \cite{MR2721464}.

\begin{definition}
An edge $[u,v]\in \oT$ is a \emph{$k$-relevant} edge if $k<|v-u|\leq kn.$ 
\end{definition}

\begin{definition}
An edge $[u,v]\in \oT$ is a \emph{$k$-boundary} edge if $|v-u|=k.$ 
\end{definition}

\begin{definition}

An \emph{angle} of a $k$-triangulation $\oT$ on $\ocalCn$ is given by a pair of edges $[u,v],[v,w]\in \oT$ such that $u \prec v \prec w$ and there exists no edge $[v,w_0]$ in $\oT$ with $w_0 \prec u \prec v \prec w.$ Such an angle is denoted by $\angle(u,v,w).$ An angle $\angle(u,v,w)$ of $\oT$ is called \emph{$k$-relevant} if at least one of $[u,v]$ and $[v,w]$ is $k$-relevant with length $<kn$. 
\end{definition}

According to this definition, every $k$-relevant angle contains either two $k$-relevant edges or a $k$-relevant edge and a $k$-boundary edge.

\begin{definition}
    Given an angle $\angle(u,v,w)$ of $\oT$. An edge $[v,w_0]$ with $w_0 \prec u \prec v \prec w$ is called an \emph{angle bisector} of $\angle(u,v,w)$; or equivalently, we say $[v,w_0]$ \emph{bisects} the angle $\angle(u,v,w)$. 
    
\end{definition} 

An edge $[v,w_0]$ might also be an angle bisector of more than one angle, or an angle bisector of an angle contained in a collection of edges; in these cases we abuse terminology and just call $[v,w_0]$ an angle bisector.

\section{Multitriangulations on the half-cylinder}\label{sec:multi}

In this section, we prove a series of properties about multitriangulations on the half-cylinder. The main result is \Cref{2starparty}, which generalizes the Star Decomposition Theorem for $n$-gons (\cite[Theorem~4.1]{MR2721464}) to $2$-triangulations on the half-cylinder with $n$ marked points.

In the following lemmas, let $T$ denote a $k$-triangulation on $\calCn$, which lifts to $\pi^{-1}(T)=\oT$ on $\ocalCn$.

\begin{lemma}\label{length>kn}
    A $k$-triangulation $T$ of $\calCn$ contains no edge of length $\ell > kn$.
\end{lemma}

\begin{proof}
    Suppose, for contradiction, that $T$ contains an edge $e = [v, v+\ell]$ with $\ell > kn$. Consider its preimage in $\ocalCn$. Then the edges
    \[
        e,\ e+n,\ e+2n,\ \dots,\ e+kn
    \]
    all appear in $\oT$, and their endpoints satisfy
    \[
        v < v+n < \dots < v+kn < v+\ell < v+\ell+n < \dots < v+\ell+kn,
    \]
    so these $k+1$ edges form a $(k+1)$-crossing. This contradicts the assumption that $\oT$ is $(k+1)$-crossing-free.
\end{proof}

\begin{lemma}\label{e!knedge}
    Every $k$-triangulation $T$ of $\calCn$ contains exactly one edge of length $kn$.
\end{lemma}

\begin{proof}
    To prove existence, suppose, for contradiction, that $T$ contains no edge of length $kn$. Let $e = [v, v+\ell]$ be an edge in $\oT$ of maximal length, and let $f = [v, v+kn]$. We claim that the set of edges $\oT \cup \{f + tn \mid t \in \mathbb{Z}\}$
    is $(k+1)$-crossing-free.

    Suppose not. Let $j \geq 1$ be the minimal integer such that there exists a $(k+1)$-crossing in this set involving exactly $j$ distinct translations of $f$ (including $f$ itself). Denote these as
    \[
        f, f + a_1n, f + a_2n, \dots, f + a_{j-1}n,
    \]
    where $0 < a_1 < \dots < a_{j-1}$. Let $E$ be the corresponding $(k+1)$-crossing.

    Without loss of generality, we may assume that no edge in $E$ has its right endpoint strictly to the right of $f + a_{j-1}n + kn$. Otherwise, for each such edge $h \in E$, replacing $h$ with $h - kn$ yields another $(k+1)$-crossing (still involving $j$ or fewer translations of $f$), and we may instead consider this new crossing.

    Now, observe that replacing $f + a_{j-1}n$ with $e + a_{j-1}n$ yields another $(k+1)$-crossing in $\oT \cup \{f + tn \mid t \in \mathbb{Z}\}$. If this were not the case, then there would exist an edge $g = [a, b] \in E$ intersecting $f + a_{j-1}n$ but not intersecting $e + a_{j-1}n$. We consider two cases:

    Case 1: $a < v + a_{j-1}n < v + a_{j-1}n + \ell < b < v + a_{j-1}n + kn$.  
      This contradicts the assumption that $e$ is the longest edge in $\oT$.

    Case 2: $v + a_{j-1}n + \ell < a < v + a_{j-1}n + kn < b$.  
      This contradicts our assumption that no edge in $E$ has its right endpoint beyond $f + a_{j-1}n + kn$.

    Therefore, replacing $f + a_{j-1}n$ with $e + a_{j-1}n$ gives another $(k+1)$-crossing involving only $j-1$ translations of $f$, contradicting the minimality of $j$. This proves the existence of an edge of length $kn$ in $T$.

    \medskip

    To prove uniqueness, suppose $T$ contains two distinct edges of length $kn$. Then $\oT$ contains edges $e = [v, v+kn]$ and $f = [v+r, v+r+kn]$ with $0 \leq r < n$. The edges
    \[
        e, f, e+n, e+2n, \dots, e+(k-1)n
    \]
    then form a $(k+1)$-crossing, contradicting the assumption that $\oT$ is $(k+1)$-crossing-free. Hence, $T$ contains exactly one edge of length $kn$.
\end{proof}

In the remainder of this section, we consider only $2$-triangulations on $\calCn.$

\begin{lemma}\label{lemma:philipslemma}
Let $E=\{[a_1,b_1],[a_2,b_2],[a_3,b_3]\}$ be a $3$-crossing labeled such that \[a_1<a_2<a_3<b_1<b_2<b_3.\] Suppose that for $i=1$ or $i=2,$ the edge $[a_{i+1},b_{i+1}]$ is a translation of $[a_{i},b_{i}].$ Then $\pi^{-1}(\pi(E))$ contains a $3$-crossing that includes exactly one translation of $[a_i,b_i].$

\end{lemma}

\begin{proof}
We proceed assuming that $[a_2,b_2]$ is a translation of $[a_1,b_1].$ A similar proof applies in the case that $[a_3,b_3]$ is a translation of $[a_2,b_2]$ are translations.
By \Cref{length>kn}, we know that $a_2-a_1=b_2-b_1=n$ and that $[a_3,b_3]$ is not a translation of $[a_1,b_1]$ and $[a_2,b_2].$

We claim that the edges $[a_1,b_1]$, $[a_3-n,b_3-n]$, and $[a_3,b_3]$ form a $3$-crossing.
Indeed, we have
\[
a_1 = a_2-n < a_3-n<a_3,
\]
and
\[
b_1 = b_2-n < b_3-n<b_3.
\]
Since $[a_3-n,b_3-n]$ is a translation of $[a_3,b_3]$, these three edges are all contained in $\pi^{-1}(\pi(E))$, and among them only one is a translation of $[a_1,b_1],$ as required.
\end{proof}

In the following lemmas, fix a $2$-triangulation $T$ of $\calCn$ which lifts to a $2$-triangulation $\pi^{-1}(T)=\oT$ on $\ocalCn$ and a $2$-relevant angle $\angle(u,v,w)$ of $\oT.$ 
 We say an edge \emph{intersects} $\angle(u,v,w)$ if it intersects both $[u,v]$ and $[v,w].$ One could note that there always exists a $2$-boundary edge intersecting $\angle (u,v,w).$
The following definitions are consistent with \cite{MR2721464} and \cite{PHDThesis}; we repeat them here for clarity. We say an edge $e=[a,b]$ is $v$-\emph{farther} than $f=[c,d]$ if $u\prec a \preceq c \prec v \prec d \preceq b \prec w$ and $e\neq f.$ We say $e$ is $v$-\emph{maximal} if there does not exist any edge intersecting $\angle(u,v,w)$ that is $v$-farther than $e.$

Let $e=[a,b]$ be the $v$-maximal edge intersecting $\angle{(u,v,w)}$ with order $u \prec a \prec v \prec b \prec w$. We work in the next two lemmas towards showing that $[u,b]$ and $[a,w]$ are edges in $\oT.$

\begin{lemma}\label{lemma:lessthan2n}
The edges $[u,b]$ and $[a,w]$ have length at most $2n.$ 
\end{lemma}

\begin{proof}
We only prove that $|b-u|\leq 2n,$ as the argument showing $|w-a|\leq 2n$ is similar. There are three possible orderings of the vertices $u$, $v$, and $w$:

Case 1: $u < v < w$\\
Then $u<a<v<b<w.$ Suppose for contradiction that $b-u<2n.$ We claim that edges $[a-2n,b-2n],$ $[v-2n,w-2n],$ and $[u,v]$ form a $3$-crossing. Indeed, \[
a-2n < v-2n < u < b-2n < w-2n < v.
\] The first and fourth inequalities follow from the fact that $[a,b]$ and $[v,w]$ intersect. The second and fifth follow from \Cref{length>kn}. The third is our assumption.

Case 2: $v<w<u$\\

Then $v<b<w<u.$ By \Cref{length>kn}, we directly obtain
\[
u-b < u - v \leq 2n.
\]

Case 3: $w < u < v$\\
Then we must have $w < u < a < v$, and there are two possible positions for $b:$ either $b<w$ or $v<b.$ If $b<w,$ then $u-b<a-b\leq 2n$ by \Cref{length>kn}. If $v<b,$ suppose for contradiction that $b-u>2n.$ We claim that the edge $[a-2n,-2n]$ intersects $\angle(u,v,w)$ $v$-farther than $[a,b],$ contradicting the $v$-maximality of $[a,b].$ Indeed, $[a-2n,b-2n]$ and $[u,v]$ intersect because
\[a-2n<v-2n<u<b-2n<a<v.\]
Here, the second and fourth inequalities follow from \Cref{length>kn}, and the third from our assumption. Similarly, $[a-2n,b-2n]$ and $[v,w]$ intersect because 
\[a-2n<v-2n<w<u<b-2n<a<v.\] Hence, $[a-2n,b-2n]$ intersects $\angle(u,v,w)$ and is $v$-farther than $[a,b],$ contradicting the $v$-maximality of $[a,b].$ 
\end{proof}

\begin{lemma}\label{lemma:3crossing2crossing}
Assume that there exists a $2$-crossing $F = \{f_1, f_2\} \subset \oT\cup \pi^{-1}(\pi([u,b]))$ such that $F\cup\{[u,b]\}$ is a $3$-crossing and $F$ contains a translation of $[u,b]$. Then there exists a $2$-crossing $F' \subset \oT$ intersecting $[u,b]$ such that $F' \cup\{[u,b]\}$ is a $3$-crossing.
\end{lemma}

\begin{proof}
We use $\alpha_\infty$, the point at infinity on $\ocalCn:$ by placing $\alpha_\infty$ in the cyclic order of a set of points, we can determine their ordering. To begin, note that by \Cref{lemma:lessthan2n}, the edge $[u,b]$ has length at most $2n$, so $F$ must contain at least one edge in $\oT$.

Without loss of generality, let $f_1 = [c_1,d_1]$ and $f_2 = [c_2,d_2]$, with cyclic order 

\[
b \prec c_1 \prec c_2 \prec u \prec d_1 \prec d_2.
\]
First, consider the case where $f_2 = [v,d_2]$ is a translation of $[u,b_1]$ that bisects $\angle(u,v,w)$. Observe that if $u\prec \alpha_\infty\prec v$ or $b\prec \alpha_\infty \prec d_2,$ then $v$ is a translation of $u$ and $d_2$ is a translation of $v.$ As such, $\{f_1, f_2, [u,b]\}$ forms a $3$-crossing satisfying the hypotheses of \Cref{lemma:philipslemma}, which in turn yields our desired $2$-crossing $F'.$ This leaves two remaining configurations: $v\prec \alpha_\infty\prec b$ and $d_2\prec\alpha_\infty \prec u,$ which we consider next.

If $v\prec \alpha_\infty\prec b,$ then $v=u+n.$ In the case that $u-n<c_1,$ we have a $3$-crossing $\{f_1,[u,u-n]\}$ satisfying the claim because $[u,u-n]=[v-n,u-n]$ is an edge of $\oT.$ Now assuming $u-n>c_1,$ we have 
\[b<c_1<u-n<c_1+n<u<d_1<d_1+n\]
and the $2$-crossing $\{f_1,f_1+n\}$ satisfies the claim.

If $d_1\prec \alpha_{\infty}\prec u,$ then $v=u+n.$ In the case that $d_1-n<v,$ we have a $2$-crossing $\{d_1,[v,v+n]\}$ satisfying the claim since $[v,v+n]=[u+n,v+n]$ is an edge of $\oT.$ Now assuming $d_1-n\leq v,$ because $v+n<d_1<d_2,$ we have $v<d_1-n<d_2-n=b$ as well. This gives 
\[c_1-n<u<c_1<v<d_1-n<b<d_1\] and so the $2$-crossing $\{f_1,f_1-n\}$ satisfies the claim.

As $\oT$ is $3$-crossing-free, a subset of a translation of any $3$-crossing in $\oT\cup F$ must satisfy the claim. We proceed assuming that $\oT\cup F$ is $3$-crossing-free. First consider the position of $d_2$ in the cyclic order. If $w\prec d_2\prec u,$ we need only consider the case that $c_2\neq v.$ Then either $u\prec c_2\prec v$ and $\{f_1,f_2,[u,v]\}$ forms a $3$-crossing, or $v\prec c_2\prec b_1$, in which case $\{f_2,[a,b],[v,w]\}$ forms a $3$-crossing. Consequently, we assume $b\prec d_1\prec d_2\preceq w.$ Then if $a\prec c_1\prec c_2\prec b,$ $\{f_2,[a,b],[u,b]\}$ is a $3$-crossing. Thus, we have reduced to the case that $[c_1,d_1]$ intersects $\angle(u,v,w)$ positioned $v$-farther than $[a,b].$ By the $v$-maximality of $[a,b],$ $[c_1,d_1]$ must be a translation of $[u,b_1].$ We proceed to show the claim is satisfied in this final case.

Note that if $d_1\prec \alpha_\infty\prec u$ or $d_1\prec \alpha_\infty\prec c_1,$ then $d_1$ is a translation of $b$ and $c_1$ is a translation of $u.$ As such, the $3$-crossing $\{[v,w],[u,b],[c_1,d_1]\}$ satisfies the conditions of \Cref{lemma:philipslemma}, which in turn yields a $2$-crossing that satisfies the claim. If $u\prec \alpha_\infty \prec c_1,$ then $u-n=d$ and $v-n<c_1=b_1-n,$ so the $2$-crossing $\{[v,w],[v-n,u-n]\}$ satisfies the claim. If $v-n\leq w,$ then since $a<v,$ we also have $a-n<v-n<w.$ Thus, $[a-n,b-n]$ intersects $\angle(u,v,w)$ positioned $v$-farther than $[a,b],$ contradicting the $v$-maximality of $[a,b].$ 
This completes the proof.
\end{proof}

\begin{theorem}[Star Decomposition Theorem]\label{2starparty}
Let $T$ be a $2$-triangulation on $\calCn$ which lifts to $\oT$ on $\ocalCn$, then any $2$-relevant angle $\angle (u,v,w)$ in $\oT$ belongs to a unique $2$-star in $\oT.$
\end{theorem}

\begin{proof} Let $e=[a,b] \in \oT$ denote the unique $v$-maximal edge intersecting $\angle (u,v,w)$ labeled such that $u\prec a \prec v \prec b \prec w.$  Suppose that $[u,b]\notin \oT,$ then there exists a $2$-crossing $\{f_1,f_2\}\subset \oT\cup \pi^{-1}(\pi([u,b]))$ that prevents the edge $[u,b]$ (that is, $\{f_1, f_2, [u,b]\}$ forms a $3$-crossing). By \Cref{lemma:lessthan2n}, we know that at least one of $f_1$ or $f_2$ lies in $\oT$, as otherwise $|u-b|>2n$. Then by \Cref{lemma:3crossing2crossing}, we may additionally assume that both $f_1$ and $f_2$ lie in $\oT.$
Similar to the proof of \cite[Theorem~4.1]{MR2721464}, we can show that $[u,b]$ (and likewise $[a,w]$) is contained in $\oT,$ and moreover show that $\oT$ contains the $2$-star determined by vertices $u\prec a \prec v \prec b \prec w.$ Uniqueness follows from the definition of angles and $k$-stars.
\end{proof}

\begin{corollary}[Maximal Lifting Theorem]\label{maximalityfromstars}
Given a $2$-triangulation $T$ on $\calCn$ which lifts to $\oT$ on $\ocalCn,$
let $e$ be an edge on $\ocalCn$ of length $\le 2n$ that is not contained in $\oT$. Then $\oT\cup \{e\}$ contains a $3$-crossing.
\end{corollary}
\begin{proof}
Set $e=[v_1,v_2].$ Since $|v_2-v_1|\le2n,$ $e$ bisects an angle $\angle (u,v_1,w)$ of $\oT$ such that at least one of $[u,v_1]$ or $[v_1,w]$ has length $<2n.$ Then by \Cref{2starparty}, $\angle (u,v_1,w)$ is contained in a $2$-star determined by vertices $u\prec a\prec v_1\prec b \prec w.$ Since $[v_1,v_2]$ bisects $\angle (u,v_1,w),$ $\{[v_1,v_2],[a,w],[u,b]\}$ gives our desired $3$-crossing.    
\end{proof}

We conjecture that the Star Decomposition Theorem (\Cref{2starparty}) and the Maximal Lifting Theorem (\Cref{maximalityfromstars}) on the half-cylinder can be generalized for arbitrary $k$ as follows:

\begin{conjecture}\label{conj:stardecompk}
Let $T$ be a $k$-triangulation on $\calCn$ which lifts to $\oT$ on $\ocalCn$, then any $k$-relevant angle $\angle (u,v,w)$ in $\oT$ belongs to a unique $k$-star in $\oT.$
\end{conjecture}

\begin{conjecture}\label{conj:maxliftingk}
    Given a $k$-triangulation $T$ on $\calCn$ which lifts to $\oT$ on $\ocalCn,$ let $e$ be an edge on $\ocalCn$ of length $\le kn$ that is not contained in $\oT$. Then $\oT\cup \{e\}$ contains a $(k+1)$-crossing.
\end{conjecture}

Furthurmore, we conjecture that analogous results hold for the $k$-triangulations on all finite-type surfaces as well. Following \cite[Section~1.2~of~Part~III]{PHDThesis}, we view any surface $\calS$ as the quotient of the hyperbolic disc $\D$ by a Fuchsian group $\Gamma$, where $\D=\ocalS$ is the universal cover of $\calS$. In this setting we likewise define the \emph{length} of an edge on $\D$ (note that this length may not always be finite). Consistent with \Cref{convention:exclude_short_edge}, we exclude all edges of length $<k$ when discussing a $k$-triangulation on $\ocalS$ (or $\calS$).

\begin{conjecture}[Star Decomposition Conjecture]\label{conj:othersurface_stardecompk}
    Given a finite-type surface $\calS$ and its universal cover $\pi:\ocalS\rightarrow\calS$. Let $T$ be a $k$-triangulation on $\calS$ which lifts to $\oT$ on $\ocalS.$ Let $\angle (u,v,w)$ be an angle in $\oT$. Suppose that $[u,v]$ and $[v,w]$ are not the preimage of the same edge $e\in T$ under $\pi$, then $\angle (u,v,w)$ belongs to a unique $k$-star contained in $\oT.$ 
\end{conjecture}

\begin{conjecture}[Maximal Lifting Conjecture]\label{conj:othersurface_maxliftingk}[cf.\ \cite[Remark~12~of~Part~III]{PHDThesis} and \Cref{remark:revisedstrong}]
    Given a finite-type surface $\calS$ and its universal cover $\pi:\ocalS\rightarrow\calS$. Let $T$ be a $k$-triangulation on $\calS$ which lifts to $\oT$ on $\ocalS.$ Let $e$ be an edge on $\ocalS$ such that $e$ is not contained in $\oT$ and $\pi^{-1}(\pi(e))$ does not contain a $(k+1)$-crossing. Then, $\oT\cup \{e\}$ contains a $(k+1)$-crossing.
\end{conjecture}

\section{Bijection: $2$-triangulations on the half-cylinder and $n$-periodic $2$-triangulations on the $4n$-gon}\label{sec:bij}

\begin{definition}
    An \emph{$n$-periodic} $k$-triangulation on a $2kn$-gon is a $k$-triangulation on the $2kn$-gon that is invariant under rotation by $\frac{2\pi}{2k} = \frac{\pi}{k}$ radians.
\end{definition}

An example of a $3$-periodic $2$-triangulation on the $12$-gon was given in \Cref{3periodic2tri}.

\begin{theorem}
\label{thm:annular-triangulations-as-periodic}
    There is a canonical bijection between $2$-triangulations on $\calCn$ and $n$-periodic $2$-triangulations on the $4n$-gon.
\end{theorem}

\begin{proof}

Let $E$ be a set of diagonals on $\calCn$ which lifts to $\overline{E}$ on $\ocalCn.$ Label the vertices of the $4n$-gon counterclockwise by $\alpha_{[i]}$ where $[i]$ is the congruence class of $i$ modulo $4n.$ We define $\phi(E) = \{[\alpha_{[i]},\alpha_{[j]}] \mid [\alpha_i,\alpha_j]\in \overline{E}\}.$ We observe that $\phi$ yields a bijection between subsets of edges of $\calCn$ and $n$-periodic subsets of edges on the $4n$-gon, namely, \ a subset of edges on the $4n$-gon invariant under rotation by $\frac{\pi}{2}$. 

First, we show that $\phi$ and $\phi^{-1}$ preserve $3$-crossings.  Note that a $3$-crossing $E$ on $\calCn$ is the projection of a $3$-crossing $\{[u_1,v_1],[u_2,v_2],[u_3,v_3]\}$ on $\ocalCn$ with vertices labeled such that $u_1\prec u_2 \prec u_3 \prec v_1 \prec v_2 \prec v_3.$ Since the greatest pairwise distance between vertices in $\{u_1,u_2,u_3,v_1,v_2,v_3\}$ is at most $4n,$ the given cyclic orientation is preserved by $\phi\circ \pi.$  As information about intersections is uniquely given by cyclic order on the endpoints, $\phi(E)$ is a $3$-crossing on the $4n$-gon. Similarly, a $3$-crossing $F$ on the $4n$-gon is mapped to a unique $3$-crossing $\phi^{-1}(E)=F$ on $\calCn.$ Thus, a collection of edges $E$ on $\calCn$ is $3$-crossing-free if and only if $\phi(E)$ is $3$-crossing-free.

Then we show that $\phi$ restricted to $k$-triangulations yields our desired bijection.
Consider a $k$-triangulation $T$ on $\calCn.$ For any edge $e$ on $\ocalCn$ not contained in $\oT$ 
{that has length $\le 2n$,} \Cref{maximalityfromstars} shows that $\oT\cup\{e\}$ contains a $3$-crossing. Thus, $\phi(T)$ is a $k$-triangulation on the $4n$-gon. Conversely, as adding an edge to an $n$-periodic $2$-triangulation on the $4n$-gon creates a $3$-crossing, we conclude that $\phi^{-1}$ sends $k$-triangulations to $k$-triangulations.
\end{proof}

The size of the orbit of an edge or $k$-star on the $2kn$-gon under the rotation action of $\mathbb{Z}/2k\ZZ$ is equal to the size of its preimage under $\phi^{-1}.$ For the below proofs, note that the size of the orbit of an edge $e$ is $2k$ if $e$ has length $<kn$ and is $k$ if $e$ has length $kn.$ Additionally, the size of the orbit of any $k$-star is $2k.$ Using the bijection outlined in \Cref{thm:annular-triangulations-as-periodic} and this information about orbits, we can determine the number of $2$-stars, $2$-relevant edges, and edges in a given $2$-triangulation on $\ocalCn.$

As shown in \cite{MR2721464}, every $k$-triangulation on the $2kn$-gon has exactly $2kn-2k$ $k$-stars, $k(2kn-2k-1)$ $k$-relevant edges, and $k(4kn-2k-1)$ edges. The following counts then follow directly from the bijection in \Cref{thm:annular-triangulations-as-periodic}.

\begin{corollary}\label{Cnpurity}
Any $2$-triangulation on $\calCn$ contains exactly $n-1$ $2$-stars, $2(n-1)$ $2$-relevant edges, and $2(2n-1)$ total edges.
\end{corollary}

 We additionally note that the proof of \Cref{thm:annular-triangulations-as-periodic} can be easily generalized to a conditional proof for $k > 2$, assuming that \Cref{conj:stardecompk} holds. Accordingly, we propose the following conjectures.

\begin{conjecture}\label{conj:bijection-for-general-k}[cf.\ \cite[Remark~12~of~Part~III]{PHDThesis} and \Cref{remark:revisedstrong}]
    There is a bijection between $k$-triangulations on $\calCn$ and $n$-periodic $k$-triangulations on the $2nk$-gon.
\end{conjecture}

\begin{conjecture}\label{generalkpurity}[cf.\ \cite[Theorem~1.2.5~and~Corollary~1.2.6~of~Part~III]{PHDThesis}]
Any $k$-triangulation on $\calCn$ contains exactly $n-1$ $k$-stars, $k(n-1)$ $k$-relevant edges, and $k(2n-1)$ total edges.
\end{conjecture}

\section{Chevron pipe dreams}\label{sec:chev}

In the 2010s, Vincent Pilaud and Chrisitan Stump independently established a connection between pipe dreams and multitriangulations \cite[Section~4.1.4]{PilaudThesis}, \cite[Theorem~2.1]{STUMP20111794} \cite[Section~7]{PP12}. Starting from a $k$-triangulation on the $n$-gon, the following is the construction of the pipe dream in the $(n-1)\times (n-1)$  staircase polyomino shape given by them:
Label the rows $n,n-1,\dots, 2$ from top to bottom and the columns $1,2,\dots, n-1$ from left to right.
In each box $(i,j)$ where there is an edge between vertices $i$ and $j$ in the $k$-triangulation, place a bump tile \begin{mosaic}{0.5}
    \tiJF \\
\end{mosaic} \ . In all other boxes, 
place a crossing tile \begin{mosaic}{0.5}
    \tiX \\
\end{mosaic}. \ This process yields a natural bijection between $k$-triangulations on the $n$-gon and reduced pipe dreams for the permutation $\pi_{n,k} := [1, \dots, k, n - k, n - k - 1, \dots, k + 1] \in \mathfrak{S}_{n - k}.$ Equivalently, these are pipe dreams in which no pair of pipes crosses twice and the pipe which enters the left boundary in the $i$-th row from the top exits the top boundary in the $\pi_{n,k}(i)$-th column from the left. For clarity, we call them \emph{staircase pipe dreams}.

In \Cref{chevronconstruction}, we construct \emph{chevron pipe dreams} from staircase pipe dreams associated with $k$-triangulations. This new model maintains key properties of the staircase pipe dream while better capturing possible symmetries of $k$-triangulations, allowing us to show that edges of $2$-triangulations on the half-cylinder can be flipped uniquely in \Cref{sec:regular}. The readers could verify that a similar idea also appears in \cite[Figure~1.12 of~~Part~III]{PHDThesis}.

In the following text, we abbreviate northeast as NE, southeast as SE, southwest as SW, and northwest as NW.

\begin{definition}\label{chevronconstruction}
The \emph{chevron pipe dream} associated with a $k$-triangulation on the $2n$-gon, denoted $\Chev_{2n,k},$ is obtained by applying the following cutting and gluing steps to the staircase pipe dream:

\begin{itemize}
    
    \item[Step 1:] Remove the NW-most $k$ pipes and delete any tiles emptied by this operation. From the NE region of the diagram, separate the largest possible inverted pyramid from the remainder of the pipe dream. The separated inverted pyramid has have $2n-5$ tiles in its top row.
    \item[Step 2:] Replace the separated inverted pyramid with its mirror image, and then rotate it $\frac{\pi}{2}$ radians clockwise.
    \item[Step 3:] Reattach the separated pyramid so that the base of the pyramid is centered on the left edge of the remainder. This attachment guarantees that the endpoint of a pipe on the top endpoint of a pipe on the staircase pipe dream reattaches to the endpoint of the same pipe on the leftmost boundary. 
    \item[Step 4:] 
    The figure from Step 3 should have two boxes in its topmost row and two boxes in its rightmost column. Draw a vertical line separating the two topmost boxes, and a horizontal line separating the two rightmost boxes. These two lines separate out a triangle in the NE region of the diagram. Replace the separated triangle with its mirror image.
    \item[Step 5:] Reattach the separated triangle along the SW-boundary of the remaining shape so that the bottom edge of the leftmost square coincides with the top edge of the triangle. This procedure guarantees that the endpoint on the NE boundary of a pipe on the shape resulting from Step 3 reattaches to the endpoint of the same point on the SW boundary.
    
\end{itemize}
For clarity, we additionally illustrate the result of each of these steps as applied to a $3$-periodic $2$-triangulation on the $12$-gon.
\end{definition}

\begin{figure}[H]
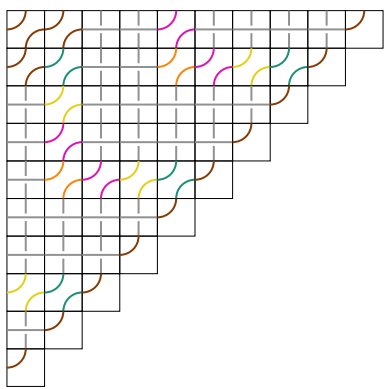
\label{fig:regular-pd}
\centering
    \begin{mosaic}{0.5}
    \Sbrown\& \Sbrown\& \tiX\& \tiX\& \Spink\& \tiX\& \tiX\& \tiX\& \tiX\& \tiJ \\ 
    \Sbrown\& \Steal\& \tiX \& \tiX\& \Sorange\& \Spink\& \Syellow\& \Steal\& \tiJ\\
    \tiX\& \Syellow\& \tiX\& \tiX\& \tiX\& \tiX\& \tiX\& \tiJ \\
    \tiX\& \Spink\& \tiX\& \tiX\& \tiX\& \tiX\& \tiJ \\
    \tiX\& \Sorange\& \Spink\& \Syellow\& \Steal\& \tiJ \\
    \tiX\& \tiX\& \tiX\& \tiX\& \tiJ \\ 
    \tiX\& \tiX\& \tiX\& \tiJ \\ 
    \Syellow\& \Steal\& \tiJ \\
    \tiX\& \tiJ \\
    \tiJ \\
    \end{mosaic}
    \caption{The staircase pipe dream corresponding to the $3$-periodic $2$-triangulation on a $12$-gon in \Cref{3periodic2tri}.}
\end{figure}

\begin{figure}[H]
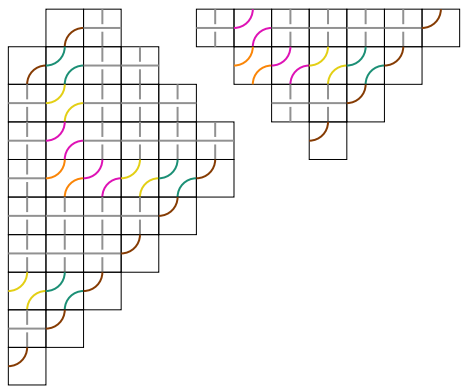

    \centering
    \begin{mosaic}{0.5}
    \& \tiF\& \tiX\& \& \& \tiX\& \Spink\& \tiX\& \tiX\& \tiX\& \tiX\& \tiJ \\ 
    \tiF\& \Steal\& \tiX \& \tiX\& \& \& \Sorange\& \Spink\& \Syellow\& \Steal\& \tiJ  \\ 
    \tiX\& \Syellow\& \tiX\& \tiX\& \tiX\& \& \& \tiX\& \tiX\& \tiJ  \\ 
    \tiX\& \Spink\& \tiX\& \tiX\& \tiX\& \tiX\& \& \& \tiJ  \\
    \tiX\& \Sorange\& \Spink\& \Syellow\& \Steal\& \tiJ  \\
    \tiX\& \tiX\& \tiX\& \tiX\& \tiJ  \\ 
    \tiX\& \tiX\& \tiX\& \tiJ  \\ 
    \Syellow\& \Steal\& \tiJ  \\
    \tiX\& \tiJ  \\
    \tiJ  \\
\end{mosaic}
    \caption{Step 1: The NW-most $2$ pipes are pruned and the corner box is removed. The NE inverted pyramid is separated from the remainder of the pipe dream}
    \label{fig:prune}
\end{figure}

\begin{figure}[H]
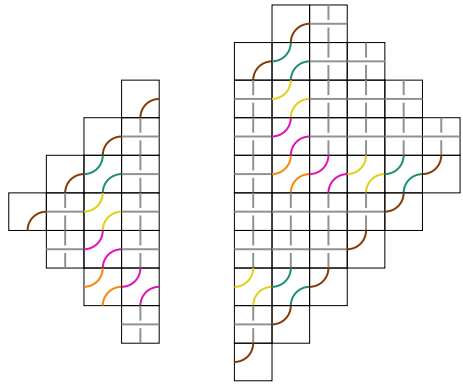

    \centering
    \begin{mosaic}{0.5}
\& \& \& \& \& \&   \& \tiF\& \tiX \& \& \& \& \& \& \& \& \\
\& \& \& \& \& \&  \tiF\&\Steal\&\tiX\&\tiX \& \& \& \& \& \\
    \& \& \& \tiF \& \& \& \tiX\& \Syellow\& \tiX\& \tiX\& \tiX \& \& \\
    \& \& \tiF \&\tiX \& \& \& \tiX\& \Spink\& \tiX\& \tiX\& \tiX\& \tiX\& \\
    \& \tiF \&\Steal \&\tiX \& \& \& \tiX\& \Sorange\& \Spink\& \Syellow\& \Steal\& \tiJ\\
    \tiF \&\tiX \&\Syellow \&\tiX \& \& \& \tiX\& \tiX\& \tiX\& \tiX\& \tiJ \\
    \& \tiX \&\Spink \&\tiX \& \& \& \tiX\& \tiX\& \tiX\& \tiJ\\
    \& \& \Sorange \&\Spink \& \& \& \Syellow\& \Steal\& \tiJ  \\
    \& \& \& \tiX \& \& \& \tiX\& \tiJ\\ 
    \& \& \& \& \& \&  \tiJ \\
\end{mosaic}
    \caption{Step 2: The separated inverted pyramid is replaced with its mirror image and rotated.}
    \label{fig:rotate and flip}
\end{figure}
\begin{figure}[H]
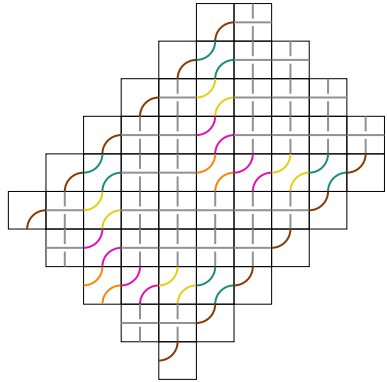

    \centering
    \begin{mosaic}{0.5}
\& \& \& \&    \& \tiF\& \tiX \& \& \& \& \& \& \& \& \\
\& \& \& \&   \tiF\&\Steal\&\tiX\&\tiX \& \& \& \& \& \\
    \& \& \& \tiF\& \tiX\& \Syellow\& \tiX\& \tiX\& \tiX \& \& \\
    \& \& \tiF \&\tiX  \& \tiX\& \Spink\& \tiX\& \tiX\& \tiX\& \tiX\& \\
    \& \tiF \&\Steal \&\tiX  \& \tiX\& \Sorange\& \Spink\& \Syellow\& \Steal\& \tiJ\\
    \tiF \&\tiX \&\Syellow \&\tiX  \& \tiX\& \tiX\& \tiX\& \tiX\& \tiJ \\
    \& \tiX \&\Spink \&\tiX  \& \tiX\& \tiX\& \tiX\& \tiJ\\
    \& \& \Sorange \&\Spink  \& \Syellow\& \Steal\& \tiJ  \\
    \& \& \& \tiX  \& \tiX\& \tiJ\\ 
    \& \& \&  \&  \tiJ \\
\end{mosaic}
    \caption{Step 3: The separated pyramid is reattached.}
    \label{fig:reattach}
\end{figure}

\begin{figure}[H]
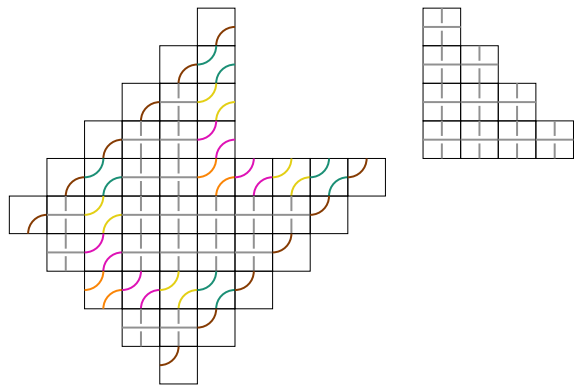

    \centering
    \begin{mosaic}{0.5}
\& \& \& \&    \& \tiF\& \& \& \& \& \& \tiX \& \& \& \& \& \& \& \& \\
\& \& \& \&   \tiF\&\Steal\& \& \& \& \& \& \tiX\&\tiX \& \& \& \& \& \\
    \& \& \& \tiF\& \tiX\& \Syellow\& \& \& \& \& \& \tiX\& \tiX\& \tiX \& \& \\
    \& \& \tiF \&\tiX  \& \tiX\& \Spink\& \& \& \& \& \& \tiX\& \tiX\& \tiX\& \tiX\& \\
    \& \tiF \&\Steal \&\tiX  \& \tiX\& \Sorange\& \Spink\& \Syellow\& \Steal\& \tiJ\\
    \tiF \&\tiX \&\Syellow \&\tiX  \& \tiX\& \tiX\& \tiX\& \tiX\& \tiJ \\
    \& \tiX \&\Spink \&\tiX  \& \tiX\& \tiX\& \tiX\& \tiJ\\
    \& \& \Sorange \&\Spink  \& \Syellow\& \Steal\& \tiJ  \\
    \& \& \& \tiX  \& \tiX\& \tiJ\\ 
    \& \& \&  \&  \tiJ \\
\end{mosaic}
    \caption{Step 4: The NE triangle is separated and replaced with its mirror image.}
    \label{fig:cut-again}
\end{figure}

\begin{figure}[H]
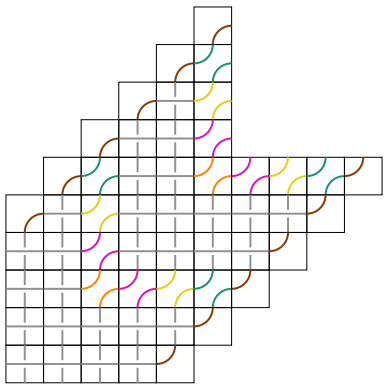

    \centering
    \begin{mosaic}{0.5}
    \& \& \& \& \& \tiF\& \& \& \& \\
    \& \& \& \& \tiF\& \Steal \& \& \& \& \\
    \& \& \& \tiF\& \tiX\& \Syellow\& \& \& \& \\
    \& \& \tiF\& \tiX\& \tiX\& \Spink\& \& \& \& \\
    \& \tiF\& \Steal\& \tiX\& \tiX\& \Sorange\& \Spink\& \Syellow\& \Steal\& \tiJ \\
    \tiF\& \tiX\& \Syellow\& \tiX\& \tiX\& \tiX\& \tiX\& \tiX\& \tiJ \\
    \tiX\& \tiX\& \Spink\& \tiX\& \tiX\& \tiX\& \tiX\& \tiJ \\
    \tiX\& \tiX\& \Sorange\& \Spink\& \Syellow\& \Steal\& \tiJ \\
    \tiX\& \tiX\& \tiX\& \tiX\& \tiX\& \tiJ \\
    \tiX\& \tiX\& \tiX\& \tiX\& \tiJ \\
\end{mosaic}

    \caption{Step 5: The separated triangle is reattached.}
    \label{fig:reattach2}
\end{figure}

\begin{lemma}
The chevron pipe dream $\Chev_{2n,k}$ is reduced, and moreover, each pair of pipes crosses exactly once.

\end{lemma}

\begin{proof}
    The cutting and gluing procedure outlined in \Cref{chevronconstruction} consists of an initial pruning and two cutting and gluing operations. As the staircase pipe dream is reduced and has permutation $\pi_{2n,k},$ after the initial pruning each pair of pipes crosses exactly once. During the cutting and gluing operations, we reattach a separated component in such a way that endpoints of the same pipes are glued. Thus, the construction preserves pipes and their crossing information.
\end{proof}

As is the case in \cite{STUMP20111794}, each pipe in our chevron pipe dream $\Chev_{2n,k}$ corresponds to a $k$-star in the $k$-triangulation on the $2n$-gon.
We use notation for $k$-stars in the $k$-triangulation to refer to pipes in the chevron pipe dream.

The staircase pipe dream corresponding to a $k$-triangulation on a $2n$-gon has rows labeled $2n,2n-1,\dots, 1$ from top to bottom and columns $1,2,\dots, 2n-1$ from left to right. Extending this to all integers so that all rows and columns of the infinite grid are labeled induces a labeling of tiles on the chevron pipe dream. For the purposes of this labeling, we consider the separated inverted pyramid and triangle to move on the grid while the remainder of the pipe dream is static for the two cutting and gluing operations. 

A key property of the chevron polyomino shape is its symmetry about the reflection axis, which we formally define here:

\begin{definition}
    The \emph{reflection axis} of a chevron pipe dream $\Chev_{2n,k}$ is the axis of symmetry of the chevron polyomino shape. The reflection axis runs SW to NE through tiles labeled $[i,i+n].$ 
\end{definition}

Let $e=[i,j]$ denote an edge of the $2n$-gon and $e+n$ its image under rotation by $\pi.$ We notice that on the chevron pipe dream, the tile corresponding to $e+n$ is the image of the tile corresponding to $e$ after reflecting across the reflection axis. 

The chevron shape was chosen additionally as to be consistent with the bijection between $k$-triangulations on the $2n$-gon invariant under rotation by $\pi$ radians and facets of the multi-cluster complex $\Delta_c^k(B_{n-k})$ as shown in \cite[Theorem~2.10]{CeballosLabbeStump}. We note the following.

\begin{proposition}
A $k$-triangulation on the $2n$-gon is invariant under rotation by $\pi$ if and only if its associated chevron pipe dream is symmetric with respect to the reflection axis. 
\end{proposition}

\begin{definition} \label{periodicpipedef}We say a chevron pipe dream corresponding to a $k$-triangulation on the $2kn$-gon" is \emph{$n$-periodic} if all $(i_1,j_1),(i_2,j_2)$ with $i_1\equiv i_2+\ell n \mod 2kn$ and $j_1\equiv 
 j_2+\ell n \mod 2kn$ for some integer $\ell$ have the same filling (both \begin{mosaic}{0.5}
    \tiJF \\
\end{mosaic} \ or both \begin{mosaic}{0.5}
    \tiX \\
\end{mosaic}). 
\end{definition}

   \begin{lemma}\label{periodicbijection}

A $k$-triangulation on the $2nk$-gon is $n$-periodic if and only if the associated chevron pipe dream under \Cref{chevronconstruction} is $n$-periodic.
   \end{lemma}
   \begin{proof}
This follows from a simple verification that the two cutting and gluing operations of \Cref{periodicpipedef} applied to a staircase pipe dream for a $k$-triangulation on the $2kn$-gon preserve labeling of tiles, modulo $2kn.$ In particular, a tile $(i,j)$ on the staircase pipe dream either remains static or is moved to another tile $(i',j'),$ where $i'=i+\ell\cdot(2kn)$ and $j'=j+\ell\cdot(2kn).$ Thus, $n$-periodicity on $k$-triangulations on the $2nk$-gon is exactly equivalent to $n$-periodicity on the chevron pipe dream. The bijection follows.
    \end{proof}

\section{The flip graph of rotationally symmetric $2$-triangulations}\label{sec:regular}

Let $\tilde{E}_{2kn,n}$ denote the set of equivalence classes of edges of the $2kn$-gon of the convex $n$-gon after identification via rotation by $\frac{\pi}{k}.$ Each element of $\tilde{E}_{2kn,n}$ is an $n$-periodic edge, namely, a set $\tilde{e}$ containing an edge $e$ and its rotational copies by integer multiples of $\frac{\pi}{k}.$ We identify an $n$-periodic $k$-triangulation on the $2kn$-gon $T$ with a subset $\tilde{T}\subset \tilde{E}_{2kn,n}.$ 
When $k=2$, elements of $\tilde{E}_{4n,n}$ correspond to edges of $\calCn$ as described in \Cref{thm:annular-triangulations-as-periodic}, and $\tilde T$ corresponds to a $2$-triangulation on $\calCn.$

By \Cref{periodicbijection}, any $\tilde{e}\in \tilde{E}_{2kn,n}$ naturally corresponds to a set of tiles on the chevron polyomino $\Chev_{2kn,k}$ which respect the $n$-periodicity conditions outlined in \Cref{periodicpipedef}.

In the following, let $T$ be an $n$-periodic 2-triangulation on the $4n$-gon, and let $\tilde{e} \in \tilde{T}$ be a set of edges containing a 2-relevant edge $e$ and its rotational copies, which we call a $2$-relevant $n$-periodic edge.

\begin{figure}[H]
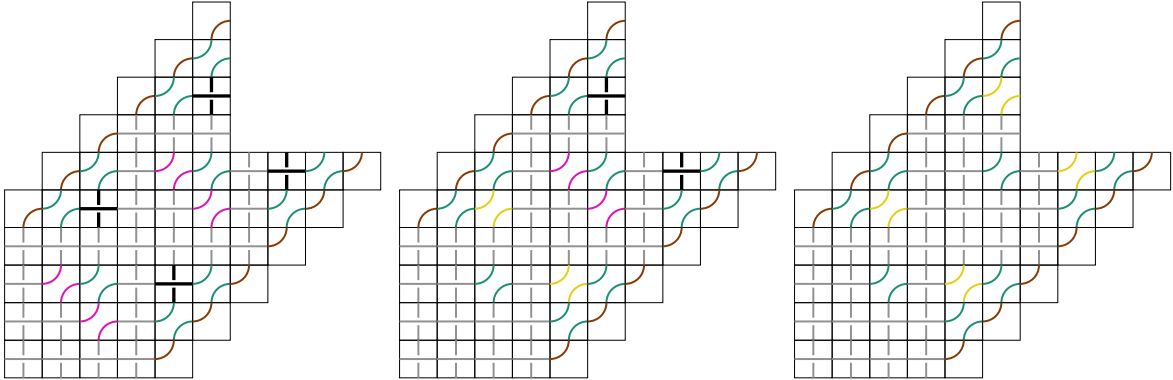

    \centering
   \begin{mosaic}{0.5}
    \& \& \& \& \& \tiF \& \& \& \& \\ 
    \& \& \& \& \tiF \& \Steal \& \& \& \& \\ 
    \& \& \& \tiF \& \Steal \& \tiXthick \& \& \& \&\\ 
    \& \& \tiF\& \tiX\& \tiX\& \tiX\& \& \& \& \& \\ 
    \& \tiF\& \Steal\& \tiX\& \Spink\& \Steal\& \tiX\& \tiXthick\& \Steal\& \tiJ \\ 
    \tiF\& \Steal\& \tiXthick\& \tiX\& \tiX\& \Spink\& \tiX\& \Steal\& \tiJ \\ 
    \tiX\& \tiX\& \tiX\& \tiX\& \tiX\& \tiX\& \tiX\& \tiJ \\
    \tiX\& \Spink\& \Steal\& \tiX\& \tiXthick\& \Steal\& \tiJ\& \\ 
    \tiX\& \tiX\& \Spink\& \tiX\& \Steal\& \tiJ\& \\ 
    \tiX\& \tiX\& \tiX\& \tiX\& \tiJ \\ 
    \end{mosaic}\begin{mosaic}{0.5}
    \& \& \& \& \& \tiF\& \& \& \& \\ 
    \& \& \& \& \tiF\& \Steal\& \& \& \& \\ 
    \& \& \& \tiF\& \Steal\& \tiXthick\& \& \& \&\\ 
    \& \& \tiF\& \tiX\& \tiX\& \tiX\& \& \& \& \& \\ 
    \& \tiF\& \Steal\& \tiX\& \Spink\& \Steal\& \tiX\& \tiXthick\& \Steal\& \tiJ \\ 
    \tiF\& \Steal\& \Syellow\& \tiX\& \tiX\& \Spink\& \tiX\& \Steal\& \tiJ \\ 
    \tiX\& \tiX\& \tiX\& \tiX\& \tiX\& \tiX\& \tiX\& \tiJ \\
    \tiX\& \tiX\& \Steal\& \tiX\& \Syellow\& \Steal\& \tiJ\& \\ 
    \tiX\& \tiX\& \tiX\& \tiX\& \Steal\& \tiJ\& \\ 
    \tiX\& \tiX\& \tiX\& \tiX\& \tiJ \\ 
    \end{mosaic} \begin{mosaic}{0.5}
    \& \& \& \& \& \tiF\& \& \& \& \\ 
    \& \& \& \& \tiF\& \Steal\& \& \& \& \\ 
    \& \& \& \tiF\& \Steal\& \Syellow\& \& \& \&\\ 
    \& \& \tiF\& \tiX\& \tiX\& \tiX\& \& \& \& \& \\ 
    \& \tiF\& \Steal\& \tiX\& \tiX\& \Steal\& \tiX\& \Syellow\& \Steal\& \tiJ \\ 
    \tiF\& \Steal\& \Syellow\& \tiX\& \tiX\& \tiX\& \tiX\& \Steal\& \tiJ \\ 
    \tiX\& \tiX\& \tiX\& \tiX\& \tiX\& \tiX\& \tiX\& \tiJ \\
    \tiX\& \tiX\& \Steal\& \tiX\& \Syellow\& \Steal\& \tiJ\& \\ 
    \tiX\& \tiX\& \tiX\& \tiX\& \Steal\& \tiJ\& \\ 
    \tiX\& \tiX\& \tiX\& \tiX\& \tiJ \\ 
    \end{mosaic}
    \caption{(Left) The chevron pipe dream associated with a $3$-periodic $2$-triangulation $T$ of the $12$-gon. We color $\tilde e$ in pink and $\tilde{f}$ in bold black. The pipe that passes through the NW bump of $e_0$ passes through the SW bump of $e_0+n.$ (Middle) The chevron pipe dream associated with $(T\setminus \{e_0,e_0+2n\})\cup\{f_0,f_0+2n\}.$ Since $f$ lies SW of $e_0+n$ and $f+2n$ lies SW of $e_0+2n,$ the pipes which bump at $e_0+n$ and $e_0+3n$ still cross at $f+n$ and $f+3n$ respectively. (Right) The chevron pipe dream associated with $(T\setminus \tilde e)\cup \tilde f.$ \label{3sharks}}

\end{figure}

\begin{figure}[H]
    \centering
   \begin{tikzpicture}[auto=center,scale=2]

    \draw[line width=0.5mm,steal] (0, 1) -- (0, -1);
    \draw[line width=0.5mm,steal] (1, 0) -- (-1, 0);

    \draw[line width=0.5mm,steal] (0, 1) -- (1, 0);
    \draw[line width=0.5mm,steal] (1, 0) -- (0, -1);
    \draw[line width=0.5mm,steal] (0, -1) -- (-1, 0);
    \draw[line width=0.5mm,steal] (-1, 0) -- (0, 1);

    \draw[line width=0.5mm,steal] (0.866, 0.5) -- (0.5, -0.866);
    \draw[line width=0.5mm,steal] (0.5, -0.866) -- (-0.866, -0.5);
    \draw[line width=0.5mm,steal] (-0.866, -0.5) -- (-0.5, 0.866);
    \draw[line width=0.5mm,steal] (-0.5, 0.866) -- (0.866, 0.5);

    \draw[line width=0.5mm,sgrey] (0.866, 0.5) -- (0.866, -0.5);
    \draw[line width=0.5mm,sgrey] (-0.5, 0.866) -- (0.5, 0.866);
    \draw[line width=0.5mm,sgrey] (-0.866, 0.5) -- (-0.866, -0.5);
    \draw[line width=0.5mm,sgrey] (-0.5, -0.866) -- (0.5, -0.866);

    \draw[line width=0.5mm,sgrey] (0.5, 0.866) -- (1, 0);
    \draw[line width=0.5mm,sgrey] (-0.866, 0.5) -- (0, 1);
    \draw[line width=0.5mm,sgrey] (-0.5, -0.866) -- (-1, 0);
    \draw[line width=0.5mm,sgrey] (0, -1) -- (0.866, -0.5);

    \draw[line width=0.5mm,sgrey] (0, 1) -- (0.866, 0.5);
    \draw[line width=0.5mm,sgrey] (-1, 0) -- (-0.5, 0.866);
    \draw[line width=0.5mm,sgrey] (0, -1) -- (-0.866, -0.5);
    \draw[line width=0.5mm,sgrey] (0.5, -0.866) -- (1, 0);

    \draw[line width=0.5mm,sgrey] (0.866, 0.5) -- (0.5, 0.866);
    \draw[line width=0.5mm,sgrey] (0.5, 0.866) -- (0, 1);
    \draw[line width=0.5mm,sgrey] (0, 1) -- (-0.5, 0.866);
    \draw[line width=0.5mm,sgrey] (-0.5, 0.866) -- (-0.866, 0.5);
    \draw[line width=0.5mm,sgrey] (-0.866, 0.5) -- (-1, 0);
    \draw[line width=0.5mm,sgrey] (-1, 0) -- (-0.866, -0.5);
    \draw[line width=0.5mm,sgrey] (-0.866, -0.5) -- (-0.5, -0.866);
    \draw[line width=0.5mm,sgrey] (-0.5, -0.866) -- (0, -1);
    \draw[line width=0.5mm,sgrey] (0, -1) -- (0.5, -0.866);
    \draw[line width=0.5mm,sgrey] (0.5, -0.866) -- (0.866, -0.5);
    \draw[line width=0.5mm,sgrey] (0.866, -0.5) -- (1, 0);
    \draw[line width=0.5mm,sgrey] (1, 0) -- (0.866, 0.5);

    \draw[line width=0.5mm,spink] (0.866, 0.5) -- (-1, 0);
    \draw[line width=0.5mm,spink] (1, 0) -- (-0.866, -0.5);
    \draw[line width=0.5mm,spink] (0.5, -0.866) -- (0, 1);
    \draw[line width=0.5mm,spink] (0, -1) -- (-0.5, 0.866);

    \draw[thick,red] (0.866, 0.5) node  {$\bullet$};
    \draw[thick,Green] (0.5, 0.866) node {$\bullet$};
    \draw[thick,blue] (0, 1) node {$\bullet$};
    \draw[thick,red] (-0.5, 0.866) node {$\bullet$};
    \draw[thick,Green] (-0.866, 0.5) node {$\bullet$};
    \draw[thick,blue] (-1, 0) node {$\bullet$};
    \draw[thick,red] (-0.866, -0.5) node {$\bullet$};
    \draw[thick,Green] (-0.5, -0.866) node {$\bullet$};
    \draw[thick,blue] (0, -1) node {$\bullet$};
    \draw[thick,red] (0.5, -0.866) node {$\bullet$};
    \draw[thick,Green] (0.866, -0.5) node {$\bullet$};
    \draw[thick,blue] (1, 0) node {$\bullet$};
\end{tikzpicture}
\begin{tikzpicture}[auto=center,scale=2]

    \draw[line width=0.5mm,steal] (0, 1) -- (0, -1);
    \draw[line width=0.5mm,steal] (1, 0) -- (-1, 0);

    \draw[line width=0.5mm,steal] (0, 1) -- (1, 0);
    \draw[line width=0.5mm,steal] (1, 0) -- (0, -1);
    \draw[line width=0.5mm,steal] (0, -1) -- (-1, 0);
    \draw[line width=0.5mm,steal] (-1, 0) -- (0, 1);

    \draw[line width=0.5mm,steal] (0.866, 0.5) -- (0.5, -0.866);
    \draw[line width=0.5mm,steal] (0.5, -0.866) -- (-0.866, -0.5);
    \draw[line width=0.5mm,steal] (-0.866, -0.5) -- (-0.5, 0.866);
    \draw[line width=0.5mm,steal] (-0.5, 0.866) -- (0.866, 0.5);

    \draw[line width=0.5mm,sgrey] (0.866, 0.5) -- (0.866, -0.5);
    \draw[line width=0.5mm,sgrey] (-0.5, 0.866) -- (0.5, 0.866);
    \draw[line width=0.5mm,sgrey] (-0.866, 0.5) -- (-0.866, -0.5);
    \draw[line width=0.5mm,sgrey] (-0.5, -0.866) -- (0.5, -0.866);

    \draw[line width=0.5mm,sgrey] (0.5, 0.866) -- (1, 0);
    \draw[line width=0.5mm,sgrey] (-0.866, 0.5) -- (0, 1);
    \draw[line width=0.5mm,sgrey] (-0.5, -0.866) -- (-1, 0);
    \draw[line width=0.5mm,sgrey] (0, -1) -- (0.866, -0.5);

    \draw[line width=0.5mm,sgrey] (0, 1) -- (0.866, 0.5);
    \draw[line width=0.5mm,sgrey] (-1, 0) -- (-0.5, 0.866);
    \draw[line width=0.5mm,sgrey] (0, -1) -- (-0.866, -0.5);
    \draw[line width=0.5mm,sgrey] (0.5, -0.866) -- (1, 0);

    \draw[line width=0.5mm,sgrey] (0.866, 0.5) -- (0.5, 0.866);
    \draw[line width=0.5mm,sgrey] (0.5, 0.866) -- (0, 1);
    \draw[line width=0.5mm,sgrey] (0, 1) -- (-0.5, 0.866);
    \draw[line width=0.5mm,sgrey] (-0.5, 0.866) -- (-0.866, 0.5);
    \draw[line width=0.5mm,sgrey] (-0.866, 0.5) -- (-1, 0);
    \draw[line width=0.5mm,sgrey] (-1, 0) -- (-0.866, -0.5);
    \draw[line width=0.5mm,sgrey] (-0.866, -0.5) -- (-0.5, -0.866);
    \draw[line width=0.5mm,sgrey] (-0.5, -0.866) -- (0, -1);
    \draw[line width=0.5mm,sgrey] (0, -1) -- (0.5, -0.866);
    \draw[line width=0.5mm,sgrey] (0.5, -0.866) -- (0.866, -0.5);
    \draw[line width=0.5mm,sgrey] (0.866, -0.5) -- (1, 0);
    \draw[line width=0.5mm,sgrey] (1, 0) -- (0.866, 0.5);

    \draw[line width=0.5mm,spink] (0.866, 0.5) -- (-1, 0);
    \draw[line width=0.5mm,spink] (1, 0) -- (-0.866, -0.5);

    \draw[line width=0.5mm,syellow] (-0.866, -0.5) -- (0, 1);
    \draw[line width=0.5mm,syellow] (0.866, 0.5) -- (0, -1);

    \draw[thick,red] (0.866, 0.5) node  {$\bullet$};
    \draw[thick,Green] (0.5, 0.866) node {$\bullet$};
    \draw[thick,blue] (0, 1) node {$\bullet$};
    \draw[thick,red] (-0.5, 0.866) node {$\bullet$};
    \draw[thick,Green] (-0.866, 0.5) node {$\bullet$};
    \draw[thick,blue] (-1, 0) node {$\bullet$};
    \draw[thick,red] (-0.866, -0.5) node {$\bullet$};
    \draw[thick,Green] (-0.5, -0.866) node {$\bullet$};
    \draw[thick,blue] (0, -1) node {$\bullet$};
    \draw[thick,red] (0.5, -0.866) node {$\bullet$};
    \draw[thick,Green] (0.866, -0.5) node {$\bullet$};
    \draw[thick,blue] (1, 0) node {$\bullet$};
\end{tikzpicture}
\begin{tikzpicture}[auto=center,scale=2]

    \draw[line width=0.5mm,steal] (0, 1) -- (0, -1);
    \draw[line width=0.5mm,steal] (1, 0) -- (-1, 0);

    \draw[line width=0.5mm,steal] (0, 1) -- (1, 0);
    \draw[line width=0.5mm,steal] (1, 0) -- (0, -1);
    \draw[line width=0.5mm,steal] (0, -1) -- (-1, 0);
    \draw[line width=0.5mm,steal] (-1, 0) -- (0, 1);

    \draw[line width=0.5mm,steal] (0.866, 0.5) -- (0.5, -0.866);
    \draw[line width=0.5mm,steal] (0.5, -0.866) -- (-0.866, -0.5);
    \draw[line width=0.5mm,steal] (-0.866, -0.5) -- (-0.5, 0.866);
    \draw[line width=0.5mm,steal] (-0.5, 0.866) -- (0.866, 0.5);

    \draw[line width=0.5mm,sgrey] (0.866, 0.5) -- (0.866, -0.5);
    \draw[line width=0.5mm,sgrey] (-0.5, 0.866) -- (0.5, 0.866);
    \draw[line width=0.5mm,sgrey] (-0.866, 0.5) -- (-0.866, -0.5);
    \draw[line width=0.5mm,sgrey] (-0.5, -0.866) -- (0.5, -0.866);

    \draw[line width=0.5mm,sgrey] (0.5, 0.866) -- (1, 0);
    \draw[line width=0.5mm,sgrey] (-0.866, 0.5) -- (0, 1);
    \draw[line width=0.5mm,sgrey] (-0.5, -0.866) -- (-1, 0);
    \draw[line width=0.5mm,sgrey] (0, -1) -- (0.866, -0.5);

    \draw[line width=0.5mm,sgrey] (0, 1) -- (0.866, 0.5);
    \draw[line width=0.5mm,sgrey] (-1, 0) -- (-0.5, 0.866);
    \draw[line width=0.5mm,sgrey] (0, -1) -- (-0.866, -0.5);
    \draw[line width=0.5mm,sgrey] (0.5, -0.866) -- (1, 0);

    \draw[line width=0.5mm,sgrey] (0.866, 0.5) -- (0.5, 0.866);
    \draw[line width=0.5mm,sgrey] (0.5, 0.866) -- (0, 1);
    \draw[line width=0.5mm,sgrey] (0, 1) -- (-0.5, 0.866);
    \draw[line width=0.5mm,sgrey] (-0.5, 0.866) -- (-0.866, 0.5);
    \draw[line width=0.5mm,sgrey] (-0.866, 0.5) -- (-1, 0);
    \draw[line width=0.5mm,sgrey] (-1, 0) -- (-0.866, -0.5);
    \draw[line width=0.5mm,sgrey] (-0.866, -0.5) -- (-0.5, -0.866);
    \draw[line width=0.5mm,sgrey] (-0.5, -0.866) -- (0, -1);
    \draw[line width=0.5mm,sgrey] (0, -1) -- (0.5, -0.866);
    \draw[line width=0.5mm,sgrey] (0.5, -0.866) -- (0.866, -0.5);
    \draw[line width=0.5mm,sgrey] (0.866, -0.5) -- (1, 0);
    \draw[line width=0.5mm,sgrey] (1, 0) -- (0.866, 0.5);


    \draw[line width=0.5mm,syellow] (-0.866, -0.5) -- (0, 1);
    \draw[line width=0.5mm,syellow] (0.866, 0.5) -- (0, -1);
    \draw[line width=0.5mm,syellow] (1, 0) -- (-0.5, 0.866);
    \draw[line width=0.5mm,syellow] (0.5, -0.866) -- (-1, 0);

    \draw[thick,red] (0.866, 0.5) node  {$\bullet$};
    \draw[thick,Green] (0.5, 0.866) node {$\bullet$};
    \draw[thick,blue] (0, 1) node {$\bullet$};
    \draw[thick,red] (-0.5, 0.866) node {$\bullet$};
    \draw[thick,Green] (-0.866, 0.5) node {$\bullet$};
    \draw[thick,blue] (-1, 0) node {$\bullet$};
    \draw[thick,red] (-0.866, -0.5) node {$\bullet$};
    \draw[thick,Green] (-0.5, -0.866) node {$\bullet$};
    \draw[thick,blue] (0, -1) node {$\bullet$};
    \draw[thick,red] (0.5, -0.866) node {$\bullet$};
    \draw[thick,Green] (0.866, -0.5) node {$\bullet$};
    \draw[thick,blue] (1, 0) node {$\bullet$};
\end{tikzpicture}
    \caption{(Left) The $3$-periodic $2$-triangulation $T$ of the $12$-gon from \Cref{3sharks}. (Middle) The $2$-triangulation $(T\setminus \{e_0,e_0+2n\})\cup \{f,f+2n\}$. (Right) The $2$-triangulation $(T\setminus \tilde e)\cup \tilde f$.}
    \label{fig:3-2triangulation}
\end{figure}
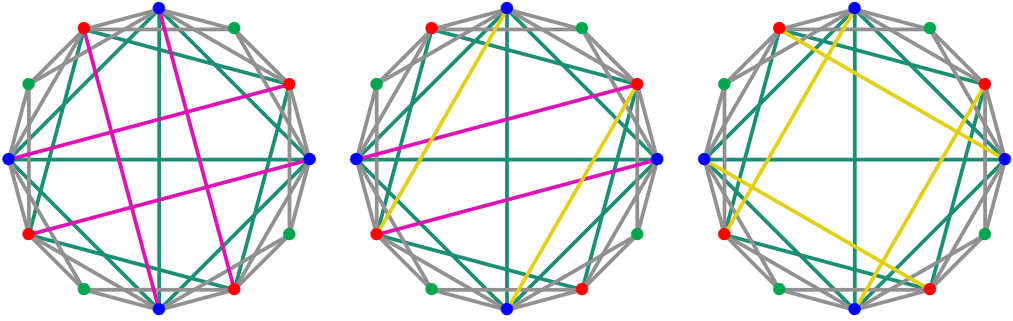

\begin{proposition}\label{flipexist}
For each 2-relevant edge $\tilde{e} \in \tilde{T}$, there exists an $n$-periodic edge $\tilde{f} \notin \tilde{T}$ such that replacing $\tilde{e}$ with $\tilde{f}$ yields another $n$-periodic 2-triangulation $(\tilde{T} \setminus \tilde{e}) \cup \tilde{f}.$
\end{proposition}

\begin{proof}
Let $e \in \tilde{e}$ lie in two 2-stars $R$ and $S$ on the $4n$-gon which share a common bisector edge $f$. First, assume that both $R$ and $S$ contain exactly one representative of $\tilde{e}$. In the Chevron pipe dream $\Chev_{4n,2}$ corresponding to $T,$ the pipes corresponding to $R$ and $S$ cross exactly once at a unique tile corresponding to an edge $f\notin T.$ Additionally, by periodicity conditions on $T$ and the corresponding chevron pipe dream, the pipes passing through the tile corresponding to each element of $\tilde{e}$ cross at a tile corresponding to an element of $\tilde{f},$ inducing a bijection between $\tilde{e}$ and $\tilde{f}$ as sets of tiles.  Thus, in the chevron pipe dream corresponding to $(\tilde{T}\setminus \tilde{e})\cup \tilde{f}$, every pair of pipes crosses exactly once, and hence $(\tilde{T}\setminus \tilde{e})\cup \tilde{f}$ is indeed a $2$-triangulation.

Now we consider the more subtle case that $R$ contains more than one edge from $\tilde{e}.$ Since the orbit of a $2$-star under the rotation action (by $\frac{\pi}{2}$) has size $4,$  $e$ cannot be of length $2n.$ Additionally, as distinct edges $e$ and $e+2n$ do not intersect or share a common vertex, they cannot belong to the same $2$-star. As $R$ contains five edges, due to periodicity, it cannot pass through the same angle twice. Accordingly, we may assume that $R$ contains $e$ and $e+n,$ and we note the following:

\begin{enumerate}
    \item  If $R$ passes through $e$ and $e+n$ on the same side of the reflection axis and passes through the NW (respectively, SE) bump of $e$, then it must pass through the SE (respectively, NW) bump of $e+n$.
    \item  If $R$ passes through $e$ and $e+n$ on opposite sides of the axis and passes through the NW (respectively, SE) bump of $e$, then it must also pass through the NW (respectively, SE) bump of $e+n$.
\end{enumerate}

Choose a representative $e_0 \in \tilde{e}$ below the reflection axis so that $e_0+n$ lies NE of $e_0$ in the pipe dream. Consider the case where $R_0$, the 2-star containing $e_0$ and $e_0+n$, passes through the NW bump of $e_0$. Then:
\begin{enumerate}
    \item $R_0$ passes through the SE bump of $e_0+n$,
    \item $R_0+n$ passes through the SE bumps of $e_0+n$ and $e_0+2n$,
    \item $R_0+2n$ passes through the SE bump of $e_0+2n$ and the NW bump of $e_0+3n$,
    \item $R_0+3n$ passes through the SE bumps of $e_0+3n$ and $e_0$.
\end{enumerate}

We obtain the following:

\begin{enumerate}
    \item $R_0$ and $R_0+3n$, which bump at $e_0$, must cross at a tile $f_0$ located NE of $e_0$ and SW of both $e_0+n$ and $e_0+3n$.
    \item $R_0+2n$ and $R_0+n,$ which bump at $e_0+2n,$ must cross at $f_0+2n$ NE of $e_0+2n$ and SW of both $e_0+n$ and $e_0+3n.$
\end{enumerate}

Similarly, $R_0$ and $R_0+n,$ which bump at $e_0+n,$ must cross at $f_0+n,$ and  $R_0+2n$ and $R_0+3n,$ which bump at $e_0+3n,$ must cross at $f_0+3n.$ Additionally, $f_0+n$ and $f_0+3n$ must either lie NE of $e_0+n$ and $e_0+3n$ respectively or SW of $e_0$ and $e_0+2n$ respectively. In either case, after replacing $e_0$ and $e_0+2n$ with crossing tiles \begin{mosaic}{0.5}
    \tiX \\
\end{mosaic} \ and $f_0$ and $f_0+2n$ with bump tiles \begin{mosaic}{0.5}
    \tiJF \\
\end{mosaic}, the pipes which bump at $e_0+n$ and $e_0+3n$ still cross at $f_0+n$ and $f_0+3n.$ Thus, $(T\setminus \tilde{e})\cup \tilde{f}$ is a $2$-triangulation.

In the case that $R_0,$ the $2$-star containing $e_0,$ passes through the SE bump of $e_0,$ a similar proof applies: the pipes which bump $e_0,e_0+n,e_0+2n,$ and $e_0+3n$ cross at $f_0,f_0+n,f_0+2n,$ and $f_0+3n$ respectively. Then $f_0$ and $f_0+2n$ must be NE of $e_0$ and $e_0+3n$ or SW of $e_0$ and $e_0+2n$ whereas $f_0+n$ and $f_0+3n$ must be NE of $e_0$ and $e_0+2n$ respectively and SW of both $e_0+n$ and $e_0+2n,$ so that after replacing $e_0$ and $e_0+2n$ with crossing tiles \begin{mosaic}{0.5}
    \tiX \\
\end{mosaic} \ and $f_0$ and $f_0+2n$ with bump tiles \begin{mosaic}{0.5}
    \tiJF \\
\end{mosaic}, the pipes which bump at $e_0+n$ and $e_0+3n$ still cross at $f_0+n$ and $f_0+3n.$ This is illustrated in \Cref{3sharks} and \Cref{3periodic2tri}.\end{proof}

\begin{proposition}\label{flipunique}

$T$ and $(T\setminus \tilde{e})\cup \tilde{f}$ are the only two $n$-periodic $2$-triangulations on the $2kn$-gon that contain $T\setminus \tilde{e}.$

\end{proposition}

\begin{proof}
As established in \Cref{flipexist}, pairs of pipes that bump at tiles in $\tilde{e}$ cross at tiles in $\tilde{f}.$  Choose a representative $e_0\in \tilde{e}$ below the reflection axis. Additionally, choose $f_0\in \tilde{f}$ so that the pipes which bump at $e_0$ cross at $f_0.$ Then $\tilde{e}=\{e_0,e_0+n,e_0+2n,e_0+3n\},$ $\tilde{f}=\{f_0,f_0+n,f_0+2n,f_0-n\},$ and for $\ell\in \{0,1,2,3\}$ pipes which bump at $e_0+\ell n$ cross at $f_0+\ell n.$ 

We can rotate $T$ as presented on the $4n$-gon so that, in the associated chevron pipe dream, no element of $\tilde{e}$ or $\tilde{f}$ lies both NE of $e_0$ and SW of $f_0,$ or both SW of $e_0$ and NE of $f_0.$ Let $R_0$ and $S_0$ denote the pipes in $T$ which bump at $e_0$ and cross at $f_0.$ For $\ell\in \{0,1,2,3\},$ the pipes $S_0+\ell n$ and $R_0+\ell_n$ bump at $e_0+\ell n$ and cross at $f_0+\ell n.$ The segments of $R_0$ and $S_0$ lying between $e_0$ and $f_0$ bound inclusively a region of the chevron pipe dream, denoted $A.$ 

Assume there is an $n$-periodic edge $\tilde{g}\notin \tilde{T}$ with $\tilde{g}\neq \tilde{f}$ such that  $(T\setminus \tilde{e})\cup \tilde{g}$ is a $2$-triangulation. If no element of $\tilde{g}$ lies on $R_0$ or $S_0$ between $e_0$ and $f_0,$ then two pipes cross twice in $(T\setminus \tilde e)\cup \tilde g$ at $e_0$ and $f_0.$ Thus, we may choose $g_0\in \tilde g$ such that $g_0+\ell n$ lies on $R_0+\ell n$ or $S_0+\ell n$ in $A,$ between $e_0+\ell n$ and $f_0+\ell n.$ Without loss of generality, let $R_0$ and $S_0$ be assigned so that $g_0+\ell n$ lies on $R_0+\ell n.$ Note that by the indexing of tiles in \Cref{periodicpipedef} , $g_0+n$ and $g_0+3n$ may not lie both NE of $e_0$ and SW of $f_0,$ or both SW of $e_0$ and NE of $f_0.$ Thus, the only elements of $\tilde{e}$ or $\tilde{g}$ which may be located in $A$ are $e_0,$ $g_0,$ and $g_0+2n.$ 

Let $R'_0$ denote the pipe that crosses $R_0$ at $g_0.$ Since $R'_0$ cannot cross $R_0$ twice, $R'_0$ must cross $S_0$ in $A.$ Let $g'$ denote the tile where $R'_0$ and $S_0$ cross. Since $g_0$ lies on $R_0$ and $R'_0,$ $g_0+2n$ cannot lie on $R_0$ or $R'_0.$ Thus, if $g_0+2n$ does not lie on $S_0,$ in $(T\setminus \tilde{e})\cup \tilde g$ a pair of pipes crosses twice at $g'$ and $f_0,$ as illustrated in \Cref{fig:unique-1}. 

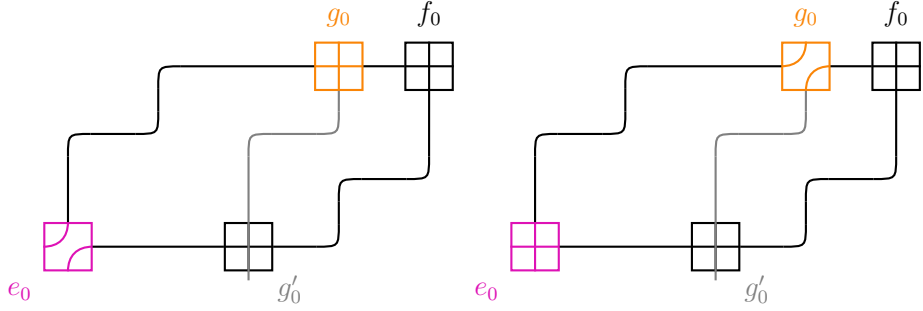
\begin{figure}[H]
    \centering

\begin{tikzpicture}[scale=1.2, every node/.style={scale=0.9},
    gridnode/.style={draw, minimum size=0.7cm, inner sep=0pt},
    cross/.style={line width=0.5pt},
    curve/.style={line width=0.5pt}
]

\newcommand{\drawfullcross}[1]{
  \draw[cross, thick] (#1.north) -- (#1.south);
  \draw[cross, thick] (#1.east) -- (#1.west);
}

\newcommand{\drawfullcrossred}[1]{
  \draw[sorange, cross, thick] (#1.north) -- (#1.south);
  \draw[sorange, cross, thick] (#1.east) -- (#1.west);
}

\newcommand{\drawcurveRightToBottomInternal}[1]{
  \path (#1.south) coordinate (bot);
  \path (#1.east) coordinate (right);
  \path (#1.center) coordinate (center);
  \draw[spink, curve, thick] 
    (right)
    .. controls ($(center)+(0.1,0)$) and ($(center)+(0,-0.1)$) ..
    (bot);
}

\newcommand{\drawcurveLeftToTopInternal}[1]{
  \path (#1.north) coordinate (top);
  \path (#1.west) coordinate (left);
  \path (#1.center) coordinate (center);
  \draw[spink, curve, thick] 
    (left)
    .. controls ($(center)+(-0.1,0)$) and ($(center)+(0,0.1)$) ..
    (top);
}

\node[gridnode, spink,  thick] (Ae) at (-1,0) {}; 
\node[gridnode, thick, draw=black] (A10) at (1,0) {}; \drawfullcross{A10} 
\node[gridnode, sorange, thick] (Ag) at (2,2) {}; \drawfullcrossred{Ag} 
\node[gridnode, thick, draw=black] (A32) at (3,2) {}; \drawfullcross{A32} 

\drawcurveRightToBottomInternal{Ae}
\drawcurveLeftToTopInternal{Ae}

\node[below left=1pt of Ae] {\textcolor{spink}{$e_0$}};
\node[above=2pt of Ag] {\textcolor{sorange}{$g_0$}};
\node[above=2pt of A32] {$f_0$};
\node[below right=-3pt of A10] {\textcolor{gray}{$g'_0$}};

\draw[thick] (Ae) -- (A10);
\draw[thick] (Ag) -- (A32);

\draw[gray, thick]
    ($(A10.south) - (0, 0.1)$) -- ($(A10.north)$);
\draw[gray, thick]
    ($(A10.north)$) -- ($(1,1)$);
\draw[gray, thick]
    ($(1,1)$) .. controls ($(1.0, 1.25)$)  .. ($(1.25,1.25)$);
\draw[gray,thick]
    ($(1.25,1.25)$) -- ($(1.75,1.25)$);
\draw[gray,thick]
    ($(1.75,1.25)$) .. controls  ($(2.0,1.25)$) .. ($(2.0,1.5)$);
\draw[gray,thick]
    ($(2.0,1.5)$) -- ($(Ag.south)$);

\draw[ thick]
    ($(Ae.north)$) -- ($(Ae) + (0,1)$);
\draw[ thick]
    ($(Ae) + (0,1)$) .. controls ($(Ae) + (0,1) + (0,0.25)$)  .. ($(Ae) + (0,1) + (0.25,0.25)$);
\draw[thick]
    ($(Ae) + (0,1) + (0.25,0.25)$) -- ($(Ae) + (0,1) + (0.25,0.25) + (0.5,0)$);
\draw[thick]
    ($(Ae) + (0,1) + (0.25,0.25) + (0.5,0)$) .. controls  ($(Ae) + (0,1) + (0.25,0.25) + (0.5,0) + (0.25,0)$) .. ($(Ae) + (0,1) + (0.25,0.25) + (0.5,0) + (0.25,0) + (0,0.25)$);
\draw[thick]
    ($(Ae) + (1,1.5)$) -- ($(Ae) + (1,1.75)$);
\draw[thick]
    ($(Ae) + (1,1.75)$) .. controls ($(Ae) + (1,1.75) + (0,0.25)$)  .. ($(Ae) + (1,1.75) + (0.25,0.25)$);
\draw[thick]
    ($(Ae) + (1.25,2) $) -- ($(Ag.west)$);

\draw[thick]
    ($(A10.east)$) -- ($(A10) + (0.75,0)$);
\draw[thick]
    ($(A10) + (0.75,0)$) .. controls ($(A10) + (0.75,0) + (0.25,0)$)  .. ($(A10) + (0.75,0) +(0.25,0.25)$);
\draw[thick]
    ($(A10) + (1.0,0.25)$) -- ($(A10) + (1.0,0.25) + (0,0.25)$);
\draw[thick]
    ($(A10) + (1.0,0.5)$) .. controls ($(A10) + (1.0,0.5) + (0,0.25)$)  .. ($(A10) + (1.0,0.5) + (0.25,0.25)$);
\draw[thick]
    ($(A10) + (1.25,0.75)$) -- ($(A10) + (1.75,0.75)$);
\draw[thick]
    ($(A10) + (1.75,0.75)$) .. controls ($(A10) + (1.75,0.75) + (0.25,0)$)  .. ($(A10) + (1.75,0.75) + (0.25,0.25)$);
\draw[thick]
    ($(A10) + (2.0,1.0)$) -- ($(A32.south)$);

\end{tikzpicture} \begin{tikzpicture}[scale=1.2, every node/.style={scale=0.9},
    gridnode/.style={draw, minimum size=0.7cm, inner sep=0pt},
    cross/.style={line width=0.5pt},
    curve/.style={line width=0.5pt}
]

\newcommand{\drawfullcross}[1]{
  \draw[cross, thick] (#1.north) -- (#1.south);
  \draw[cross, thick] (#1.east) -- (#1.west);
}

\newcommand{\drawfullcrossblue}[1]{
  \draw[spink, cross, thick] (#1.north) -- (#1.south);
  \draw[spink, cross, thick] (#1.east) -- (#1.west);
}

\newcommand{\drawcurveRightToBottomInternal}[1]{
  \path (#1.south) coordinate (bot);
  \path (#1.east) coordinate (right);
  \path (#1.center) coordinate (center);
  \draw[sorange, curve, thick] 
    (right)
    .. controls ($(center)+(0.1,0)$) and ($(center)+(0,-0.1)$) ..
    (bot);
}

\newcommand{\drawcurveLeftToTopInternal}[1]{
  \path (#1.north) coordinate (top);
  \path (#1.west) coordinate (left);
  \path (#1.center) coordinate (center);
  \draw[sorange, curve, thick] 
    (left)
    .. controls ($(center)+(-0.1,0)$) and ($(center)+(0,0.1)$) ..
    (top);
}

\node[gridnode, spink,  thick] (Ae) at (-1,0) {}; \drawfullcrossblue{Ae} 
\node[gridnode, thick, draw=black] (A10) at (1,0) {}; \drawfullcross{A10} 
\node[gridnode, sorange, thick] (Ag) at (2,2) {}; 
\node[gridnode, thick, draw=black] (A32) at (3,2) {}; \drawfullcross{A32} 

\drawcurveRightToBottomInternal{Ag}
\drawcurveLeftToTopInternal{Ag}

\node[below left=1pt of Ae] {\textcolor{spink}{$e_0$}};
\node[above=2pt of Ag] {\textcolor{sorange}{$g_0$}};
\node[above=2pt of A32] {$f_0$};
\node[below right=-3pt of A10] {\textcolor{gray}{$g'_0$}};

\draw[thick] (Ae) -- (A10);
\draw[thick] (Ag) -- (A32);

\draw[gray, thick]
    ($(A10.south) - (0, 0.1)$) -- ($(A10.north)$);
\draw[gray, thick]
    ($(A10.north)$) -- ($(1,1)$);
\draw[gray, thick]
    ($(1,1)$) .. controls ($(1.0, 1.25)$)  .. ($(1.25,1.25)$);
\draw[gray,thick]
    ($(1.25,1.25)$) -- ($(1.75,1.25)$);
\draw[gray,thick]
    ($(1.75,1.25)$) .. controls  ($(2.0,1.25)$) .. ($(2.0,1.5)$);
\draw[gray,thick]
    ($(2.0,1.5)$) -- ($(Ag.south)$);

\draw[ thick]
    ($(Ae.north)$) -- ($(Ae) + (0,1)$);
\draw[ thick]
    ($(Ae) + (0,1)$) .. controls ($(Ae) + (0,1) + (0,0.25)$)  .. ($(Ae) + (0,1) + (0.25,0.25)$);
\draw[thick]
    ($(Ae) + (0,1) + (0.25,0.25)$) -- ($(Ae) + (0,1) + (0.25,0.25) + (0.5,0)$);
\draw[thick]
    ($(Ae) + (0,1) + (0.25,0.25) + (0.5,0)$) .. controls  ($(Ae) + (0,1) + (0.25,0.25) + (0.5,0) + (0.25,0)$) .. ($(Ae) + (0,1) + (0.25,0.25) + (0.5,0) + (0.25,0) + (0,0.25)$);
\draw[thick]
    ($(Ae) + (1,1.5)$) -- ($(Ae) + (1,1.75)$);
\draw[thick]
    ($(Ae) + (1,1.75)$) .. controls ($(Ae) + (1,1.75) + (0,0.25)$)  .. ($(Ae) + (1,1.75) + (0.25,0.25)$);
\draw[thick]
    ($(Ae) + (1.25,2) $) -- ($(Ag.west)$);

\draw[thick]
    ($(A10.east)$) -- ($(A10) + (0.75,0)$);
\draw[thick]
    ($(A10) + (0.75,0)$) .. controls ($(A10) + (0.75,0) + (0.25,0)$)  .. ($(A10) + (0.75,0) +(0.25,0.25)$);
\draw[thick]
    ($(A10) + (1.0,0.25)$) -- ($(A10) + (1.0,0.25) + (0,0.25)$);
\draw[thick]
    ($(A10) + (1.0,0.5)$) .. controls ($(A10) + (1.0,0.5) + (0,0.25)$)  .. ($(A10) + (1.0,0.5) + (0.25,0.25)$);
\draw[thick]
    ($(A10) + (1.25,0.75)$) -- ($(A10) + (1.75,0.75)$);
\draw[thick]
    ($(A10) + (1.75,0.75)$) .. controls ($(A10) + (1.75,0.75) + (0.25,0)$)  .. ($(A10) + (1.75,0.75) + (0.25,0.25)$);
\draw[thick]
    ($(A10) + (2.0,1.0)$) -- ($(A32.south)$);

\end{tikzpicture}
\caption{In the diagram on the right, a pair of pipes crosses twice at $g'_0$ and $f_0.$ \label{fig:unique-1}
}
\end{figure}

We have reduced to the case where $g_0+2n$ lies on $S_0$ between $g'$ and $f_0.$ In particular, $R_0+2n$ does not cross $R'_0=S_0+2n$ in $A$ since $R_0+2n$ and $S_0+2n$ cross at $f_0+2n.$ Additionally, since $R_0+2n$ cannot cross $S_0$ twice, $R_0+2n$ and $S_0+2n$ cross at a tile $g''$ in $A.$ Thus, in $(T\setminus \tilde e)\cup \tilde g,$ a pair of pipes crosses twice at $g'$ and $g'',$ as illustrated in \Cref{fig:unique-2}.

\begin{figure}[H]
    \centering
    
\begin{tikzpicture}[scale=1.2, every node/.style={scale=0.9},
    gridnode/.style={draw, minimum size=0.7cm, inner sep=0pt},
    cross/.style={line width=0.5pt},
    curve/.style={line width=0.5pt}
]

\newcommand{\drawfullcross}[1]{
  \draw[cross, thick] (#1.north) -- (#1.south);
  \draw[cross, thick] (#1.east) -- (#1.west);
}

\newcommand{\drawfullcrossred}[1]{
  \draw[sorange, cross, thick] (#1.north) -- (#1.south);
  \draw[sorange, cross, thick] (#1.east) -- (#1.west);
}

\newcommand{\drawcurveRightToBottomInternal}[1]{
  \path (#1.south) coordinate (bot);
  \path (#1.east) coordinate (right);
  \path (#1.center) coordinate (center);
  \draw[spink, curve, thick] 
    (right)
    .. controls ($(center)+(0.1,0)$) and ($(center)+(0,-0.1)$) ..
    (bot);
}

\newcommand{\drawcurveLeftToTopInternal}[1]{
  \path (#1.north) coordinate (top);
  \path (#1.west) coordinate (left);
  \path (#1.center) coordinate (center);
  \draw[spink, curve, thick] 
    (left)
    .. controls ($(center)+(-0.1,0)$) and ($(center)+(0,0.1)$) ..
    (top);
}

\node[gridnode,  thick] (Ae) at (-1,0) {}; \drawfullcross{Ae} 
\node[gridnode, sorange, thick] (A10) at (1,0) {}; \drawfullcrossred{A10} 
\node[gridnode, sorange, thick] (Ag) at (-1,2) {}; \drawfullcrossred{Ag} 
\node[gridnode, thick, draw=black] (A32) at (1,2) {}; \drawfullcross{A32} 
\node[gridnode, thick, draw=black] (f0) at (1,3) {}; \drawfullcross{f0} 
\node[gridnode, thick, draw=black] (f02n) at (2,2) {}; \drawfullcross{f02n} 
\node[gridnode, spink, thick] (e0) at (-2,0) {};  
\node[gridnode, spink, thick] (e02n) at (-1,-1) {};  

\drawcurveRightToBottomInternal{e0}
\drawcurveLeftToTopInternal{e0}

\drawcurveRightToBottomInternal{e02n}
\drawcurveLeftToTopInternal{e02n}

\node[above right=1pt of Ae] {\textcolor{black}{$g'$}};
\node[above left=-1pt of Ag] {\textcolor{sorange}{$g_0$}};
\node[below left=-3pt of A32] {$g''$};
\node[below right=-3pt of A10] {\textcolor{sorange}{$g_0 + 2n$}};
\node[right=2pt of f0] {{$f_0$}};
\node[right=2pt of f02n] {{$f_0 + 2n$}};
\node[left=2pt of e0] {\textcolor{spink}{$e_0$}};
\node[left=2pt of e02n] {\textcolor{spink}{$e_0 + 2n$}};
--- Edges ---
\draw[thick] (e02n) -- (Ae);
\draw[thick] (Ae) -- (Ag);
\draw[thick] (A10) -- (A32);
\draw[thick] (A32) -- (f0);
\draw[thick] (e0) -- (Ae);
\draw[thick] (Ae) -- (A10);
\draw[thick] (Ag) -- (A32);
\draw[thick] (A32) -- (f02n);

\draw[thick] 
    ($(e0.north)$) -- ($(e0.north) + (0,1.47)$);
\draw[thick]
    ($(e0.north) + (0,1.47)$) .. controls ($(e0.north) + (0,1.47) + (0, 0.25)$)  .. ($(e0.north) + (0,1.47) + (0.25,0.25)$);
\draw[thick]
    ($(e0.north) + (0.25,1.72)$) -- ($(Ag.west)$);

\draw[thick] 
    ($(Ag.north)$) -- ($(Ag.north) + (0,0.5)$);
\draw[thick]
    ($(Ag.north) + (0,0.5)$) .. controls ($(Ag.north) + (0,0.5) + (0,0.25)$)  .. ($(Ag.north) + (0,0.5) + (0.25,0.25)$);
\draw[thick]
    ($(Ag.north) + (0.25,0.75)$) -- ($(f0.west)$);

\draw[thick] 
    ($(e02n.east)$) -- ($(e02n.east) + (1.48,0)$);
\draw[thick]
    ($(e02n.east) + (1.48,0)$) .. controls ($(e02n.east) + (1.48,0) + (0.25, 0)$)  .. ($(e02n.east) + (1.48,0) +(0.25,0.25)$);
\draw[thick]
    ($(e02n.east) + (1.73,0.25)$) -- ($(A10.south)$);


\draw[thick] 
    ($(A10.east)$) -- ($(A10.east) + (0.48,0)$);
\draw[thick]
    ($(A10.east) + (0.48,0)$) .. controls ($(A10.east) + (0.48,0) + (0.25,0)$)  .. ($(A10.east) + (0.48,0) + (0.25,0.25)$);
\draw[thick]
    ($(A10.east) + (0.73,0.25)$) -- ($(f02n.south)$);

\end{tikzpicture} \begin{tikzpicture}[scale=1.2, every node/.style={scale=0.9},
    gridnode/.style={draw, minimum size=0.7cm, inner sep=0pt},
    cross/.style={line width=0.5pt},
    curve/.style={line width=0.5pt}
]

\newcommand{\drawfullcross}[1]{
  \draw[cross, thick] (#1.north) -- (#1.south);
  \draw[cross, thick] (#1.east) -- (#1.west);
}

\newcommand{\drawfullcrossblue}[1]{
  \draw[spink, cross, thick] (#1.north) -- (#1.south);
  \draw[spink, cross, thick] (#1.east) -- (#1.west);
}

\newcommand{\drawcurveRightToBottomInternal}[1]{
  \path (#1.south) coordinate (bot);
  \path (#1.east) coordinate (right);
  \path (#1.center) coordinate (center);
  \draw[sorange, curve, thick] 
    (right)
    .. controls ($(center)+(0.1,0)$) and ($(center)+(0,-0.1)$) ..
    (bot);
}

\newcommand{\drawcurveLeftToTopInternal}[1]{
  \path (#1.north) coordinate (top);
  \path (#1.west) coordinate (left);
  \path (#1.center) coordinate (center);
  \draw[sorange, curve, thick] 
    (left)
    .. controls ($(center)+(-0.1,0)$) and ($(center)+(0,0.1)$) ..
    (top);
}

\node[gridnode,  thick] (Ae) at (-1,0) {}; \drawfullcross{Ae} 
\node[gridnode, sorange, thick] (A10) at (1,0) {};  
\node[gridnode, sorange, thick] (Ag) at (-1,2) {};  
\node[gridnode, thick, draw=black] (A32) at (1,2) {}; \drawfullcross{A32} 
\node[gridnode, thick, draw=black] (f0) at (1,3) {}; \drawfullcross{f0} 
\node[gridnode, thick, draw=black] (f02n) at (2,2) {}; \drawfullcross{f02n} 
\node[gridnode, spink, thick] (e0) at (-2,0) {}; \drawfullcrossblue{e0} 
\node[gridnode, spink, thick] (e02n) at (-1,-1) {}; \drawfullcrossblue{e02n} 

\drawcurveRightToBottomInternal{Ag}
\drawcurveLeftToTopInternal{Ag}

\drawcurveRightToBottomInternal{A10}
\drawcurveLeftToTopInternal{A10}

\node[above right=1pt of Ae] {\textcolor{black}{$g'$}};
\node[above left=-1pt of Ag] {\textcolor{sorange}{$g_0$}};
\node[below left=-3pt of A32] {$g''$};
\node[below right=-3pt of A10] {\textcolor{sorange}{$g_0 + 2n$}};
\node[right=2pt of f0] {{$f_0$}};
\node[right=2pt of f02n] {{$f_0 + 2n$}};
\node[left=2pt of e0] {\textcolor{spink}{$e_0$}};
\node[left=2pt of e02n] {\textcolor{spink}{$e_0 + 2n$}};
\draw[thick] (e02n) -- (Ae);
\draw[thick] (Ae) -- (Ag);
\draw[thick] (A10) -- (A32);
\draw[thick] (A32) -- (f0);
\draw[thick] (e0) -- (Ae);
\draw[thick] (Ae) -- (A10);
\draw[thick] (Ag) -- (A32);
\draw[thick] (A32) -- (f02n);

\draw[thick] 
    ($(e0.north)$) -- ($(e0.north) + (0,1.47)$);
\draw[thick]
    ($(e0.north) + (0,1.47)$) .. controls ($(e0.north) + (0,1.47) + (0, 0.25)$)  .. ($(e0.north) + (0,1.47) + (0.25,0.25)$);
\draw[thick]
    ($(e0.north) + (0.25,1.72)$) -- ($(Ag.west)$);

\draw[thick] 
    ($(Ag.north)$) -- ($(Ag.north) + (0,0.5)$);
\draw[thick]
    ($(Ag.north) + (0,0.5)$) .. controls ($(Ag.north) + (0,0.5) + (0,0.25)$)  .. ($(Ag.north) + (0,0.5) + (0.25,0.25)$);
\draw[thick]
    ($(Ag.north) + (0.25,0.75)$) -- ($(f0.west)$);

\draw[thick] 
    ($(e02n.east)$) -- ($(e02n.east) + (1.48,0)$);
\draw[thick]
    ($(e02n.east) + (1.48,0)$) .. controls ($(e02n.east) + (1.48,0) + (0.25, 0)$)  .. ($(e02n.east) + (1.48,0) +(0.25,0.25)$);
\draw[thick]
    ($(e02n.east) + (1.73,0.25)$) -- ($(A10.south)$);


\draw[thick] 
    ($(A10.east)$) -- ($(A10.east) + (0.48,0)$);
\draw[thick]
    ($(A10.east) + (0.48,0)$) .. controls ($(A10.east) + (0.48,0) + (0.25,0)$)  .. ($(A10.east) + (0.48,0) + (0.25,0.25)$);
\draw[thick]
    ($(A10.east) + (0.73,0.25)$) -- ($(f02n.south)$);

\end{tikzpicture}

\caption{\label{fig:unique-2}In the diagram on the right, a pair of pipes crosses twice at $g'$ and $g''$. }
    
\end{figure}
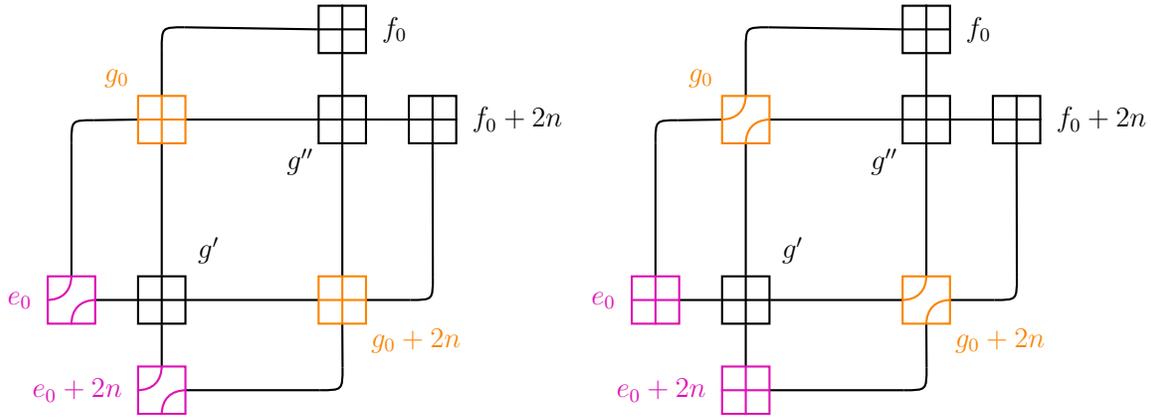

In each case, we have shown that in $(T\setminus \tilde e)\cup \tilde g,$ there exists a pair of pipes that crosses twice. Thus, $(T\setminus \tilde e)\cup \tilde g$ is not a $2$-triangulation.
\end{proof}

We say that we obtain the $n$-periodic $2$-triangulation on the $4n$-gon $(T\setminus \tilde e)\cup \tilde f$ from $T$ by flipping the $n$-periodic edge $e.$ Let $G^n_{4n,2}$ denote the \emph{flip graph} of $n$-periodic $2$-triangulations on the $4n$-gon, the graph whose vertices are $n$-periodic $2$-triangulations on the $4n$-gon and whose edges are pairs related by a flip. Equivalently, $G^n_{4n,2}$ is the graph of flips of $2$-triangulations on the half-cylinder $\calCn.$

\begin{theorem}\label{flipgraphthm}
The graph $G^n_{4n,2}$ is regular of degree $2(n-1).$
\end{theorem}

\begin{proof}

The regularity of $G^n_{4n,2}$ follows from the count of $2$-relevant edges in \Cref{Cnpurity}, in addition to the fact that every $2$-relevant $n$-periodic edge can be flipped and this flip is unique as shown in \Cref{flipexist} and \Cref{flipunique}.

\end{proof}

We additionally define a simplicial complex $\Delta_{\calCn,k}$ whose vertices are $k$-relevant edges of $\calCn$ and whose facets correspond to $k$-triangulations on $\calCn.$ \Cref{maximalityfromstars} shows that the faces of $\Delta _{\calCn,2}$ are sets of $k$-relevant edges of $\calCn$ such that the lifting to $\ocalCn$ is $(k+1)$-crossing-free. The results of this paper additionally imply the following:

\begin{corollary}\label{cor:weak pseudomanifold}
The simplicial complex $\Delta _{\calCn,2}$ is pure of dimension $2(n-1)$; moreover, it is a weak pseudomanifold (without boundary).
\end{corollary}

\begin{proof}
Purity of dimension $2(n-1)$ follows from the count of $2$-relevant edges in \Cref{Cnpurity}. As every $k$-relevant edge can be uniquely flipped as shown in \Cref{flipexist} and \Cref{flipunique}, each codimension $1$ face lies in exactly $2$ distinct facets, and thus $\Delta _{\calCn,2}$ is a weak pseudomanifold without boundary. 
\end{proof}

\begin{proof}[Proof of \Cref{thm:main}]
    This follows directly from \Cref{flipgraphthm} and \Cref{cor:weak pseudomanifold}.
\end{proof}

In \cite[Theorem~2.1]{STUMP20111794}, the simplicial complex $\Delta_{n,k},$ whose vertices are diagonals of the $n$-gon of length greater than $k$ and whose facets correspond to $k$-triangulations, is shown to be naturally isomorphic to a subword complex in type $A_{n-k-1}.$ 

In \cite[Theorem~2.10]{CeballosLabbeStump}, an analogous result is shown in type $B.$ In particular, the $B_{m-k}$ subword complex considered has vertices corresponding to a diagonal and its image under rotation by $180^\circ,$ and facets corresponding to $k$-triangulations of the $2m$-gon invariant under rotation by $180^\circ.$


We give an analogous result for the simplicial complex $\Delta_{\calCn,2}$ below. Let $Q$ be a word in the generators $S$ of a finite Coxeter group $W.$ For an element $\pi\in W,$ the subword complex $\mathcal{SC}(Q,\pi),$ which was introduced in \cite{KM05}, has faces given by subwords $P$ of $Q$ for which the complement $Q\setminus P$ contains a reduced expression of $\pi.$

\begin{proposition}\label{subwordcxunion}
    The simplicial complex $\Delta_{\calCn,2}$ is isomorphic to a union of subword complexes \[\bigcup_{w^2=w_0} \mathcal{SC}(c^{n},w),\] where $w_0$ is the longest element of the Coxeter group $B_{2n-2}$ and $c$ is the word $s_1,s_2\cdots, s_{2n-2}.$ 
\end{proposition}

\begin{proof}




As implied by \cite[Theorem~2.10]{CeballosLabbeStump}, centrally symmetric $2$-triangulations of the $4n$-gon are in bijection with subwords $P$ of the word $Q=c^{2n}$ such that $Q\setminus P$ is a reduced expression for $w_0,$ as is additionally evident from the chevron pipe dream model. Let $\sigma_i$ denote the $i$-th letter of $Q.$ From \Cref{periodicbijection} and this reading of the subword complex from the chevron pipe dream, we obtain that $n$-periodic $2$-triangulations of the $4n$-gon are in bijection with subwords $P$ of $Q$ representing $w_0$ which additionally satisfy $\sigma_i\in P\Longleftrightarrow\sigma_{n(2n-2)+i}\in P,$ for $i\in \{1,2,\dots, n(2n-2)\}.$  

By \cite[Proposition~4.4.4]{MR2133266} and \cite[Proposition~2.1~(ii)]{MR3474789}, for any element $w\in B_{2n-2}$ such that $w^2=w_0,$ the Coxeter length $\ell(w)$ is equal to $\ell(w_0)/2$. For this reason, if we write any reduced expression of $w$ twice consecutively, it would be some expression for $w_0$, and no two $w, w'$ yield the same expression for $w_0$ through this process. Therefore, there is a bijection between the reduced expressions of $w_0$ satisfying $\sigma_i\in P\Longleftrightarrow\sigma_{n(2n-2)+i}\in P,$ for $i\in \{1,2,\dots, n(2n-2)\}$ and the reduced expressions of $w$ such that $w^2=w_0$. 

As such, we additionally have that $n$-periodic $2$-triangulations of the $n$-gon are in bijection with subwords of $c^{n}$ representing an element $w$ such that $w^2=w_0,$ naturally identifies the facets of $\Delta_{\calCn,2}$ with the facets of $\cup_{w^2=w_0} \mathcal{SC}(c^{n},w).$


\end{proof}

Assuming \Cref{conj:bijection-for-general-k}, a similar proof yields the following bijection for arbitrary $k$:

\begin{conjecture}
    The simplicial complex $\Delta_{C_n,k}$ is isomorphic to a union of subword complexes \[\bigcup_{w \in \mathcal A_{k}}\mathcal{SC}(c^n,w),\] where $w_0$ is the longest element of the Coxeter group $B_{nk-k},$ $c$ is the word $s_1,s_2,\dots ,s_{nk-k},$ and $\mathcal A_w:=\{w\mid w^k=w_0,\ell(w)=\ell(w_0)/k\}\subset B_{kn-k}.$ 
\end{conjecture}
Note that when $k\geq3,$ for $w$ satisfying $w^k=w_0$, it is not automatically true that $\ell(w)=\ell(w_0)/k.$

Regarding the number of vertices in the flip graph, which is equivalent to the number of $k$-triangulations on $\calCn$, or the number of facets in $\Delta _{\calCn,k}$, we have the following conjecture.

\begin{conjecture}
    The simplicial complex $\Delta_{\calCn,k}$ has ${2(n-1)\choose (n-1)}^k$ facets.
\end{conjecture}

In \cite[Chapter 6]{Jahn}, a specific class of unions of subword complexes, called \emph{glued subword complexes}, are shown to be vertex-decomposable spheres. As we have shown that $\Delta _{\calCn,2}$ is a weak pseudomanifold -- analogous to the result for $n$-gons \cite[Theorem~1]{Jonssonfake}, we make the following conjecture.

\begin{conjecture}\label{conj:spherical 2}
The simplicial complex $\Delta _{\calCn,2}$ is a piecewise linear sphere.
\end{conjecture}

We expect the conjecture to hold for arbitrary $k.$ 

\begin{conjecture}\label{conj:spherical k}
The simplicial complex $\Delta _{\calCn,k}$ is a piecewise linear sphere.
\end{conjecture}

By \cite[Proposition~2.3]{MR2448067}, the only surfaces whose multitriangulation complexes are finite (that is, contain finitely many simplices) are polygons and half‑cylinders. Hence, an affirmative answer to \Cref{conj:spherical k}, together with \cite[Theorem~1]{Jonssonfake}, would imply that all finite multitriangulation complexes are piecewise linear spheres.





\section*{Acknowledgements}

This project was done as part of the University of Minnesota School of Mathematics Summer 2024 REU program, and was supported by NSF grant DMS-2053288 and the D.E. Shaw group. Saskia Solotko was supported by the Tufts Summer Scholars program and Katherine Tung was supported by the Harvard College Research Program for work on the project after the REU ended. We thank our mentor Pavlo Pylyavskyy and TA Joe McDonough for their guidance. We thank Ilani Axelrod-Freed, Kieran Favazza, Sergey Fomin, Carl Lian, Molly MacDonald, Vincent Pilaud, Victor Reiner, George Seelinger, and Roberta Shapiro for helpful discussions. We also thank Ayah Almousa and Victor Reiner for organizing this wonderful REU experience. 

\printbibliography

@article {CeballosLabbeStump,
    AUTHOR = {Ceballos, Cesar and Labb\'e, Jean-Philippe and Stump,
              Christian},
     TITLE = {Subword complexes, cluster complexes, and generalized
              multi-associahedra},
   JOURNAL = {J. Algebraic Combin.},
  FJOURNAL = {Journal of Algebraic Combinatorics. An International Journal},
    VOLUME = {39},
      YEAR = {2014},
    NUMBER = {1},
     PAGES = {17--51},
      ISSN = {0925-9899,1572-9192},
   MRCLASS = {05E45 (52B05)},
  MRNUMBER = {3144391},
MRREVIEWER = {Micha\l\ Adamaszek},
       DOI = {10.1007/s10801-013-0437-x},
       URL = {https://doi.org/10.1007/s10801-013-0437-x},
}

@article {PHDThesis,
    AUTHOR = {Lepoutre, Mathias},
     TITLE = {Blossoming bijections, multitriangulations : {W}hat about other surfaces?},
      YEAR = {2019},
URL={https://theses.hal.science/tel-02326948v1},
}

@incollection {MR2721464,
    AUTHOR = {Pilaud, Vincent and Santos, Francisco},
     TITLE = {Multi-triangulations as complexes of star polygons},
 BOOKTITLE = {20th {A}nnual {I}nternational {C}onference on {F}ormal {P}ower
              {S}eries and {A}lgebraic {C}ombinatorics ({FPSAC} 2008)},
    SERIES = {Discrete Math. Theor. Comput. Sci. Proc.},
    VOLUME = {AJ},
     PAGES = {319--330},
 PUBLISHER = {Assoc. Discrete Math. Theor. Comput. Sci., Nancy},
      YEAR = {2008},
   MRCLASS = {52A30 (05C10 05C62 52B05)},
  MRNUMBER = {2721464},
}

@article {Jonssonfake,
    AUTHOR = {Jonsson, Jakob},
     TITLE = {Generalized triangulations of the n-gon},
   YEAR = {2003}
}

@article{STUMP20111794,
title = {A new perspective on k-triangulations},
journal = {Journal of Combinatorial Theory, Series A},
volume = {118},
number = {6},
pages = {1794-1800},
year = {2011},
issn = {0097-3165},
doi = {https://doi.org/10.1016/j.jcta.2011.03.001},
url = {https://www.sciencedirect.com/science/article/pii/S0097316511000501},
author = {Christian Stump}
}

@article{NAKAMIGAWA2000271,
title = {A generalization of diagonal flips in a convex polygon},
journal = {Theoretical Computer Science},
volume = {235},
number = {2},
pages = {271-282},
year = {2000},
issn = {0304-3975},
doi = {https://doi.org/10.1016/S0304-3975(99)00199-1},
url = {https://www.sciencedirect.com/science/article/pii/S0304397599001991},
author = {Tomoki Nakamigawa}
}

@article {CP92,
    AUTHOR = {Capoyleas, Vasilis and Pach, J\'anos},
     TITLE = {A {T}ur\'an-type theorem on chords of a convex polygon},
   JOURNAL = {J. Combin. Theory Ser. B},
  FJOURNAL = {Journal of Combinatorial Theory. Series B},
    VOLUME = {56},
      YEAR = {1992},
    NUMBER = {1},
     PAGES = {9--15},
      ISSN = {0095-8956,1096-0902},
   MRCLASS = {52C10 (05C35)},
  MRNUMBER = {1182454},
MRREVIEWER = {L.\ M.\ Kelly},
       DOI = {10.1016/0095-8956(92)90003-G},
       URL = {https://doi.org/10.1016/0095-8956(92)90003-G},
}

@article {DKM00,
    AUTHOR = {Dress, A. and Koolen, J. H. and Moulton, V.},
     TITLE = {On line arrangements in the hyperbolic plane},
   JOURNAL = {European J. Combin.},
  FJOURNAL = {European Journal of Combinatorics},
    VOLUME = {23},
      YEAR = {2002},
    NUMBER = {5},
     PAGES = {549--557},
      ISSN = {0195-6698,1095-9971},
   MRCLASS = {52C30 (51M10)},
  MRNUMBER = {1931939},
       DOI = {10.1006/eujc.2002.0582},
       URL = {https://doi.org/10.1006/eujc.2002.0582},
}

@article {PP12,
    AUTHOR = {Pilaud, Vincent and Pocchiola, Michel},
     TITLE = {Multitriangulations, pseudotriangulations and primitive
              sorting networks},
   JOURNAL = {Discrete Comput. Geom.},
  FJOURNAL = {Discrete \& Computational Geometry. An International Journal
              of Mathematics and Computer Science},
    VOLUME = {48},
      YEAR = {2012},
    NUMBER = {1},
     PAGES = {142--191},
      ISSN = {0179-5376,1432-0444},
   MRCLASS = {52C30 (52C35 52C45 68U05)},
  MRNUMBER = {2917206},
MRREVIEWER = {Carlos\ Seara},
       DOI = {10.1007/s00454-012-9408-6},
       URL = {https://doi.org/10.1007/s00454-012-9408-6},
}

@article {CS24,
    AUTHOR = {Crespo Ruiz, Luis and Santos, Francisco},
     TITLE = {Multitriangulations and tropical {P}faffians},
   JOURNAL = {SIAM J. Appl. Algebra Geom.},
  FJOURNAL = {SIAM Journal on Applied Algebra and Geometry},
    VOLUME = {8},
      YEAR = {2024},
    NUMBER = {2},
     PAGES = {302--332},
      ISSN = {2470-6566},
   MRCLASS = {14T15 (05E45 52B40)},
  MRNUMBER = {4736289},
       DOI = {10.1137/22M1527507},
       URL = {https://doi.org/10.1137/22M1527507},
}

@article {DKKM01,
    AUTHOR = {Dress, A. and Klucznik, M. and Koolen, J. and Moulton, V.},
     TITLE = {{$2kn-\binom{2k+1}{2}$}: note on extremal combinatorics of
              cyclic split systems},
   JOURNAL = {S\'em. Lothar. Combin.},
  FJOURNAL = {S\'eminaire Lotharingien de Combinatoire},
    VOLUME = {47},
      YEAR = {2001/02},
     PAGES = {Art. B47b, 17},
      ISSN = {1286-4889},
   MRCLASS = {05D05},
  MRNUMBER = {1894022},
MRREVIEWER = {Jamie\ Simpson},
}

@article {BB93,
    AUTHOR = {Bergeron, Nantel and Billey, Sara},
     TITLE = {R{C}-graphs and {S}chubert polynomials},
   JOURNAL = {Experiment. Math.},
  FJOURNAL = {Experimental Mathematics},
    VOLUME = {2},
      YEAR = {1993},
    NUMBER = {4},
     PAGES = {257--269},
      ISSN = {1058-6458,1944-950X},
   MRCLASS = {05E99 (05E05 14M15 20C30)},
  MRNUMBER = {1281474},
MRREVIEWER = {Axel\ Kohnert},
       URL = {http://projecteuclid.org/euclid.em/1048516036},
}

@inproceedings {FK93,
    AUTHOR = {Fomin, Sergey and Kirillov, Anatol N.},
     TITLE = {The {Y}ang-{B}axter equation, symmetric functions, and
              {S}chubert polynomials},
 BOOKTITLE = {Proceedings of the 5th {C}onference on {F}ormal {P}ower
              {S}eries and {A}lgebraic {C}ombinatorics ({F}lorence, 1993)},
   JOURNAL = {Discrete Math.},
  FJOURNAL = {Discrete Mathematics},
    VOLUME = {153},
      YEAR = {1996},
    NUMBER = {1-3},
     PAGES = {123--143},
      ISSN = {0012-365X,1872-681X},
   MRCLASS = {05E05 (14C17 14M15 82B23)},
  MRNUMBER = {1394950},
MRREVIEWER = {Jean-Yves\ Thibon},
       DOI = {10.1016/0012-365X(95)00132-G},
       URL = {https://doi.org/10.1016/0012-365X(95)00132-G},
}

@article {PilaudThesis,
    AUTHOR = {Pilaud, Vincent},
     TITLE = {Multitriangulations, pseudotriangulations and some problems of realization of polytopes},
      YEAR = {2010},
URL={https://www.ub.edu/comb/vincentpilaud/documents/reports/theseVincentPilaud.pdf},
}

@article {MR2448067,
    AUTHOR = {Fomin, Sergey and Shapiro, Michael and Thurston, Dylan},
     TITLE = {Cluster algebras and triangulated surfaces. {I}. {C}luster
              complexes},
   JOURNAL = {Acta Math.},
  FJOURNAL = {Acta Mathematica},
    VOLUME = {201},
      YEAR = {2008},
    NUMBER = {1},
     PAGES = {83--146},
      ISSN = {0001-5962,1871-2509},
   MRCLASS = {57Q15 (13F60 32G15 52B70)},
  MRNUMBER = {2448067},
MRREVIEWER = {Christof\ Gei\ss},
       DOI = {10.1007/s11511-008-0030-7},
       URL = {https://doi.org/10.1007/s11511-008-0030-7},
}

@misc{vppc,
  author = "Pilaud, Vincent",
  date = "2025-06-23",
  howpublished = "Personal communication"
}

@article {MR3474789,
    AUTHOR = {Hohlweg, Christophe and Labb\'e, Jean-Philippe},
     TITLE = {On inversion sets and the weak order in {C}oxeter groups},
   JOURNAL = {European J. Combin.},
  FJOURNAL = {European Journal of Combinatorics},
    VOLUME = {55},
      YEAR = {2016},
     PAGES = {1--19},
      ISSN = {0195-6698,1095-9971},
   MRCLASS = {05E15 (06A07 20F55)},
  MRNUMBER = {3474789},
MRREVIEWER = {Marius\ T\u arn\u auceanu},
       DOI = {10.1016/j.ejc.2016.01.002},
       URL = {https://doi.org/10.1016/j.ejc.2016.01.002},
}

@book {MR2133266,
    AUTHOR = {Bj\"orner, Anders and Brenti, Francesco},
     TITLE = {Combinatorics of {C}oxeter groups},
    SERIES = {Graduate Texts in Mathematics},
    VOLUME = {231},
 PUBLISHER = {Springer, New York},
      YEAR = {2005},
     PAGES = {xiv+363},
      ISBN = {978-3540-442387; 3-540-44238-3},
   MRCLASS = {05-01 (05E15 20F55)},
  MRNUMBER = {2133266},
MRREVIEWER = {Jian-yi\ Shi},
}

@article {KM05,
    AUTHOR = {Knutson, Allen and Miller, Ezra},
     TITLE = {Gr\"obner geometry of {S}chubert polynomials},
   JOURNAL = {Ann. of Math. (2)},
  FJOURNAL = {Annals of Mathematics. Second Series},
    VOLUME = {161},
      YEAR = {2005},
    NUMBER = {3},
     PAGES = {1245--1318},
      ISSN = {0003-486X,1939-8980},
   MRCLASS = {05E15 (13C40 13F55 13P10 14M15 14N15)},
  MRNUMBER = {2180402},
MRREVIEWER = {Harry\ Tamvakis},
       DOI = {10.4007/annals.2005.161.1245},
       URL = {https://doi.org/10.4007/annals.2005.161.1245},
}

@phdthesis{Jahn,
    author = {Jahn, Dennis},
    title = {Combinatorics of subword complexes for finite type {C}oxeter systems},
    school = {Ruhr {U}niversity {B}ochum},
    year = {2022}
}

\end{document}